\documentclass[a4paper,12pt]{article}
\usepackage[T1]{fontenc}
\usepackage[english]{babel}
\usepackage{mathrsfs}

\usepackage{nonfloat}
\usepackage{amsmath}
\usepackage{amsthm}
\usepackage{amsfonts}
\usepackage{amssymb}
\usepackage{makeidx}
\usepackage[T1]{fontenc}
\usepackage[utf8]{inputenc}
\usepackage{graphicx}
\usepackage{hyperref}
\usepackage{subfigure}
\usepackage[nowrite]{frontespizio}
\usepackage{mathrsfs}
\usepackage{fancyhdr}
\usepackage{setspace}
\usepackage{verbatim}
\usepackage{mathtools}
\mathtoolsset{showonlyrefs}

\usepackage{pgfplots,pgfplotstable, relsize}
\usepgflibrary{arrows}
\usepackage{float}
\usepackage{pgf,tikz}
\usetikzlibrary{arrows}
\usepackage{caption}
\usepackage{color}

\usepackage[left=1.5cm,right=1.5cm,top=1.5cm,bottom=1.5cm]{geometry} 
\newcommand{\dis}{\displaystyle}

\newcommand{\llav}[1]{  \left\{#1\right\} }
\newcommand{\pic}[1]{  \left\langle #1\right\rangle }

\newcommand{\norm}[1]{  \left\|#1\right\| }
\newcommand{\pare}[1]{\left(#1\right)}
\newcommand{\corch}[1]{  \left[#1\right] }
\newcommand{\abs}[1]{  \left|#1\right| }
\newcommand{\CAL}[1]{\mathcal{#1}}
\newcommand{\BB}[1]{\mathbb{#1}}

\newcommand{\ds}[1]{\displaystyle{#1}}
\def\mathcolor#1#{\@mathcolor{#1}}
\def\@mathcolor#1#2#3{%
  \protect\leavevmode
  \begingroup
    \color#1{#2}#3%
  \endgroup
}

\newcommand{\blanco}[1]{\mathcolor{white}{#1}}
\def\L^*{\CAL L_{T, x^*}^{\text{max}}}

\newtheorem {theorem}{Theorem}[section]
\newtheorem{lemma}{Lemma}[section]

\newtheorem{corollario}{Corollary}[section]
\newtheoremstyle{mythm}%
{3pt}
{3pt}
{}
{}
{\bfseries}
{}
{.5em}
{}%

\theoremstyle{mythm}
 \newtheorem{remarker}{Remark}[section]

\title{Critical asymptotic behaviour in the SIR model}
\author{Monia Capanna\thanks{Universit\`a degli Studi dell'Aquila, Via Vetoio, 67100 L'Aquila, Italy. Email:\,{\tt monia.capanna@graduate.univaq.it}}}
\date{}
\begin{document}

\sloppy
\maketitle
 
\begin{abstract}
\noindent 
This article is devoted to the analysis of a particle system model for epidemics among a finite population with susceptible, infective and removed individuals (SIR). The infection mechanism depends on the relative distance between susceptibles and infected so that an infected individual is more likely to infect nearby sites than those further away. For fixed time, we prove that the density fields weakly converge to the solution of a PDE's system, as the number of particles increases. We find an implicit expression for the final survivor density of the limit equation and we analyze the asymptotics of the microscopic system, by taking first the time and after the number of particles to infinity,  showing a critical behaviour for some values of the parameters when the system is set in the mean field regime.



\end{abstract}
\section{{\bf Introduction}}
        \label{sec:1}

In \cite{wkermack27}, Kermack and McKendrick proposed the SIR model with the purpose to analyse the spread of an infection. This simple model is formulated for a population being divided into three parts, the susceptible class S, consisting of healthy people that might be infected, the infective class I, with infected individuals that can spread the disease to susceptibles,  and the recovered class R, in which there are people that had been infectious but can no longer spread or catch the disease. The general model is described by the following system of ODE's
\begin{align}\label{p85}
\begin{cases}
\frac{d}{dt}S=-\lambda SI\\
\frac{d}{dt}I=\lambda SI-\mu I\\
\frac{d}{dt}R=\mu I 
\end{cases}
\end{align}
where $\lambda$ and $\mu$ are positive constants and represent the contagious rate and  removal rate respectively. Since then several stochastic versions of the SIR model have been considered. 
In \cite{K70}, T. G. Kurtz shows that, in the mean field regime (i.e. under the assumption that each individual affects all the others with the same intensity), system \eqref{p85} arises as the limit in probability of a large population stochastic epidemic model.
%
In \cite{lalley2009} and \cite{penrose1996}, instead, the authors consider spatial variants of the stochastic SIR model, in which infectious contacts are limited to individuals that are within a fixed range from the infected one.

\vspace{+10pt}

In this article we consider a model where infected individuals have a greatest influence on the healthy ones spatially closest to them. The microscopic scenario consists of an interacting particle system, in the discrete $d$-dimensional torus, in which the possible states for the sites are $0$=susceptible, $1$=infected and $-1$=removed. Each infected site recovers with constant rate and infects the susceptible particles with a rate that depends on the relative position of infective and susceptible in such a way that nearby sites are infected with a greater intensity than those further away. We show that the model converges, in the hydrodynamic limit, to a system of PDE's which describes the evolution, in space and time, of the  macroscopic densities, $u_0(r,t)$ and $u_1\pare{r,t}$, of susceptibles and infected respectively, where $\pare{r, t}\in \mathbb T^d\times \mathbb R^+$ and $\BB T^d$ is the $d$-dimensional torus. 
Denoting by $\dis{\rho_\infty(r):=\lim_{t\to +\infty}u_0\pare{r,t}}$ and by $\rho_0(r):=u_0\pare{r,0}$ the final and
respectively initial densities of susceptibles, and by  $\rho_1\pare{r}:=u_1\pare{r, 0}$ the initial density of infected, we prove that
\begin{equation}
    \label{1.1}
\rho_\infty\pare{r}= \rho_0(r) e^{- I(r)}
     \end{equation}
where
      \begin{equation}
    \label{1.2}
I(r)= \beta\int_{\BB T^d} J\pare{r, r'}\pare{\rho_0\pare{r'}+\rho_1\pare{r'}-\rho_\infty\pare{r'}}dr'
     \end{equation}
and $ \beta J(r,r')\ge 0$ describes the infection intensity at $r$ due
to an infective individual at $r'$.
 Notice that \eqref{1.1} is not an explicit formula for the
survivors fraction $\rho_\infty\pare{r}$, as its r.h.s.\ depends on $\rho_\infty(\cdot)$ as
well. It is anyway interesting to notice that, when an external agent infects instantaneously an healthy population (i.e. $\rho_0\pare{r}+\rho_1\pare{r}\equiv 1$), the knowledge of the parameters of
the infection, $J(r, r')$ and $\beta$, allows to deduce
location and density of the initial infectors from the statistics
of the final survivors.


The macroscopic behaviour gives a good description of the microscopic one 
only for bounded range times. For this reason the hydrodynamic limit is not appropriate if we want to investigate the long time behaviour of the microscopic system.  

The main result of the paper is  given by Theorem \ref{051} in which, taking first the time and after the number of particles to infinity, we investigate the behaviour of the final microscopic density of population ever infected. The theorem is set in the mean field regime. This assumption reduces the complexity of our problem by removing the spatial dependence in the infection rate. In the theorem we show that there exists a critical value of $\beta$, given by $\beta_c=1$, below which, starting with an infinitesimal fraction of infected, the infection remains  confined. If
$\beta>1$, under further conditions on the initial data,  it spreads all over the space and the final density of survivors approaches, as the population increases, a limit value which differs significantly from the result of the hydrodynamic limit. In this case, indeed, the infection has enough time to propagate and the order of the limits in time and number of particles cannot be switched.

\section{Model definition and statements of the results}
Let $\BB T^d$ be the torus in $d$ dimensions and $\gamma^{-1}$ a positive integer. We denote by $\BB T_\gamma^d:=\pare{\gamma^{-1}\BB T^d}\cap\mathbb{Z}^d$ the microscopic version of the discrete $d-$dimensional torus. Our model is an interacting particle system in $\CAL{S}_\gamma:=\{-1,0,1\}^{\BB T_\gamma^d}$. Before describing the interaction between the particles we need to introduce a constant $\beta>0$ and a smooth function $J\in L^\infty\pare{\BB T^d\times \BB T^d; \BB R_+}$. Since the aim is to deal with a spatially homogeneous model, it is natural to assume $J(r,r')$ a symmetric function such that $J(r',r)= J(0,r'-r)$ and $\dis{\int_{\BB T^d} J(r',r) dr'}=1$ do not depend on $r$.

Let $\eta^\gamma\in \CAL S_\gamma$ be a configuration and $x$ a site in $\BB T_\gamma^d$. From now on $\eta^\gamma(x)$ denotes the value of the configuration $\eta^\gamma$ in $x$. $\eta^\gamma (x)$ can switch from $0$ to $1$ with rate $\gamma^d \beta\sum_{y\in \BB T_\gamma^d}\mathbb{I}_{\{\eta^\gamma(y)=1\}}J\pare{\gamma x,\gamma y}$, or from $1$ to $-1$ with rate $1$.
For every $\gamma>0$, and for every initial configuration $\eta^\gamma_0\in\CAL S_\gamma$, the evolution $\llav{\eta^\gamma_t}_{t\geq0}$ is a Markov process, defined in an abstract probability space $\pare{\Omega, \mathscr F, P}$. The state space of the process is $\CAL S_\gamma$ and the generator $L_\gamma$ is defined on functions $f:\CAL S_\gamma\to\mathbb{R}$ by
\begin{displaymath}
\big(L_\gamma f\big)(\eta^\gamma)=\sum_{x\in \BB T_\gamma^d}\corch{\pare{\gamma^d\beta\sum_{y\in \BB T_{\gamma}^d} J\pare{\gamma x,\gamma y}\mathbb{I}_{\{\eta^\gamma(y)=1\}}}\mathbb{I}_ {\{\eta^\gamma(x)=0\}}+\mathbb{I}_ {\{\eta^\gamma(x)=1\}}}\pare{f\pare{\eta^{\gamma, x}}-f\pare{\eta^\gamma}},
\end{displaymath}
where $\eta^{\gamma, x}$ denotes the configuration in $\CAL S_\gamma$ such that
\begin{displaymath}
\begin{split}
\eta^{\gamma, x}(y)=
\begin{cases}
1&\text{if $y=x$ and $\eta^\gamma(x)=0$},\\
-1&\text{if $y=x$ and $\eta^\gamma(x)=1$},\\
\eta^\gamma(y)&\text{otherwise}.
\end{cases}
\end{split}
\end{displaymath}
Such Markov process can be interpreted as a model for the spread of a disease. An individual on a site $x$ is healthy if $\eta^\gamma(x)=0$, infected if $\eta^\gamma(x)=1$, recovered if $\eta^\gamma(x)=-1$. Infected individuals recover with constant rate equal to 1 and are no longer able to catch the disease, while an healthy individual in $x$ becomes infected at a rate which depends on the number of sick individuals present in the system and also on their distance from $x$. 
Assuming that 
\begin{align}\label{phd}
J(0, r)\leq J(0,r')
\end{align}
for all $r, r'\in\BB R^d$ such that $\abs{r'}\leq\abs{r}$, we have the following interpretation: the more an healthy individual is close to an infected one the more he is likely to become sick.

 In order to study the hydrodynamic limit of the system, we consider the time evolution of the empirical measures $\pi_t^{\gamma, i}$  associated to the particle system:
\begin{equation}\label{0}
\begin{split}
&\pi_t^{\gamma, i}(dr)=\pi^{\gamma, i}\pare{\eta^\gamma_t, dr}=\gamma^d\sum_{x\in \BB T_\gamma^d}\mathbb{I}_{\llav{\eta^\gamma_t(x)=i}}\delta_{\gamma x}\pare{dr}\qquad i=-1,0,1
\end{split}
\end{equation}
where $\delta_{\gamma x}\pare{dr}$ is the Dirac measure on $\BB T^d$ centered in $\gamma x$. For every measure $\pi$ on the torus and for every $G\in C\pare{\BB T^d, \BB R}$ ($C\pare{\BB T^d, \BB R}$ is the set of all the continuous functions $G:\BB T^d\to\BB R$), we denote by $\pic{\pi,G}:=\int_{\BB T^d} Gd\pi$ the integral of $G$ with respect to $\pi$. Notice that
\begin{align}\nonumber
\pic{\pi_t^{\gamma, i}, G}=\gamma^d\sum_{x\in \BB T_\gamma^d}\BB{I}_{\llav{\eta^\gamma_t(x)=i}}G\pare{\gamma x}
\end{align}
is just a function of the configuration $\eta^\gamma_t$.
Since $\pi_t^{\gamma, -1}(dr)+\pi_t^{\gamma, 0}(dr)+\pi_t^{\gamma, 1}(dr)\equiv 1$, we will analyze only the evolution of $\pare{\pi_t^{\gamma, 0}(dr), \pi_t^{\gamma, 1}(dr)}_{t\geq 0}$ whose hydrodynamic behaviour is described by the following theorem. From now on we will denote by $\xrightarrow[]{P}$ the convergence in probability.
\begin{theorem}\label{teo}
Let $\rho_0\in C\pare{\BB T^d, [0,1]}$ and $\rho_1\in C\pare{\BB T^d, [0,1]}$ such that $\dis{\sup_{r\in\BB T^d}\llav{\rho_0\pare{r}+\rho_1\pare{r}}\leq 1}$. Suppose that 
\begin{align}
\limsup_{\gamma\to 0}P\pare{\abs{\pic{\pi_0^{\gamma, i}, G}-\int_{\mathbb{T}^d}\rho_i(r)G\pare{r}dr}>\varepsilon}=0,
\end{align}
for all $\varepsilon>0$,  $i\in\llav{0,1}$ and $G\in C\pare{\BB T^d, \BB R}$. Then, for every $T>0$,
\begin{equation}\label{anna}
\sup_{t\in[0,T]}\abs{\pic{\pi_t^{\gamma,i}, G}-\int_{\BB T^d}u_i\pare{t,r}G\pare{r}dr}\xrightarrow[\gamma\to 0]{P}0,
\end{equation}
where $u_0\pare{t,r}$ and $u_1\pare{t,r}$ are the solutions of 
\begin{equation}\label{2}
\begin{cases}
\frac{\partial}{\partial t} u_0(t,r)=-\beta  J* u_1(t,r) u_0(t,r)\\
\frac{\partial}{\partial t}u_1(t,r)=\beta J*u_1(t,r)u_0(t,r)-u_1(t,r)\\
u_0(0,r)=\rho_0(r), u_1(0,r)=\rho_1(r).
\end{cases}
\end{equation}
\end{theorem}
The following theorem guarantees the existence and uniqueness of the solution of the latter system and describes the behaviour of the hydrodynamic limit for large times. 
\begin{theorem}\label{secondo}
There exists a unique solution $\pare{u_0\pare{t,r}, u_1\pare{t,r}}$ of system \eqref{2} in $[0, +\infty)\times\BB T^d$. Such solution has the following properties:
\begin{itemize}
\item[(a)] $u_0\pare{t,r}$ is decreasing in $t$, moreover $0\leq u_0\pare{t,r}\leq 1$ and $0\leq u_1\pare{t,r}\leq 1-u_0\pare{t,r}$,
\item[(b)]$\dis{\lim_{t\to +\infty}\pare{u_0\pare{t,r}, u_1\pare{t,r}}=\pare{\rho\pare{r}, 0}}$, where $\rho\pare{r}$ satisfies
\begin{align}\label{seii}
\rho\pare{r}=\rho_0\pare{r}e^{-\beta J*\pare{\rho_0\pare{r}+\rho_1\pare{r}-\rho\pare{r}}}.
\end{align}
\end{itemize}
\end{theorem}

Theorem \ref{secondo} implies that  $u_0\pare{t,r}\geq \rho\pare{r}\geq \rho_0\pare{r}e^{-\beta\pare{\rho_0\pare{r}+\rho_1\pare{r}}}$ for all $t\in [0,+\infty)$.

Observe that, considering the case in which, at time $0$, an external agent infects instantaneously a fraction of an healthy population (i.e. $\rho_0\pare{r}+\rho_1\pare{r}=1$), relation \eqref{seii} allows to deduce location and density of the initial infectors from the statistics of the final survivors once that the parameters of the infection, $\beta$ and $J\pare{r,r'}$, are known.
Viceversa, if we know $J(r,r')$ and that some region was initially
free from the infection, then we can compute the strength of the
disease $\beta$ which is equal, for $r$ in such a region, to
$\dis{-\big(\int_{\BB T^d} J(r',r)[1-\rho(r')]dr'\big)^{-1}\log \rho(r)} $.

Considering now the particular case in which $J\pare{0,r}\equiv 1$ we reduce the interaction of our system in the mean field regime. Calling 
\begin{displaymath}
x^\gamma(t):=\pic{\pi_t^{\gamma,0}, 1}\qquad y^\gamma(t):=\pic{\pi_t^{\gamma,1},1}
\end{displaymath}
the stochastic averages of healthy and sick individuals respectively, by Theorem \ref{teo} and Theorem \ref{secondo}, we recover the following corollary which describes the hydrodynamic behaviour for the mean field system. Such result is known in literature and has been proved by Kurtz in \cite{K70}.
\begin{corollario}\label{teoo}
Let $\rho_0, \rho_1\in [0,1]$ such that $\rho_0+\rho_1\leq 1$ and suppose that 
\begin{align}\nonumber
\limsup_{\gamma\to 0}P\pare{\abs{x^\gamma(0)-\rho_0}>\varepsilon}=0,\quad \limsup_{\gamma\to 0}P\pare{\abs{y^\gamma(0)-\rho_1}>\varepsilon}=0,
\end{align}
for all $\varepsilon>0$. Then, for every $T>0$,
\begin{displaymath}
\sup_{t\in[0,T]}\abs{x^\gamma(t)-x(t)}\xrightarrow[\gamma\to 0]{P}0,\quad\sup_{t\in[0,T]}\abs{y^\gamma(t)-y(t)}\xrightarrow[\gamma\to 0]{P}0,
\end{displaymath}
 where $\pare{x\pare{t}, y\pare{t}}$ is the unique solution of the system

\begin{equation}\label{222}
\begin{cases}
\frac{d}{dt}x(t)=-\beta x(t)y(t)\\
\frac{d}{dt}y(t)=\beta x(t)y(t)-y(t)\\
x(0)=\rho_0,\; y(0)=\rho_1.
\end{cases}
\end{equation}
In addition $\dis{\lim_{t\to +\infty}\pare{x\pare{t}, y\pare{t}}=\pare{x_\infty,0}}$, 
where $x_\infty$ is defined as 

\begin{equation}\label{lop}
x_\infty=
\begin{cases}
\rho_0\quad\text{if }\;\rho_1=0,\\
0\quad\text{ if }\;\rho_0=0,\\
\bar x_\infty \;\; \text{if } \;\rho_1\in (0,1)\text{ and } \rho_0\neq 0,
\end{cases}
\end{equation}
with $\bar x_\infty$ the smallest solution of 
\begin{equation}\label{lk}
x=\rho_0 e^{-\beta\pare{\rho_0+\rho_1- x}}.
\end{equation}
\end{corollario}
Regimes given by \eqref{lop} follow trivially if $\rho_1=0$ or $\rho_0=0$, because the solution of \eqref{222} is given by $\pare{\rho_0,0}$ or $\pare{0, e^{-t}}$, which implies $x_\infty$ equal to $\rho_0$ or $0$ respectively. In the less trivial case in which $\rho_1\in\pare{0,1}$ and $\rho_0\neq 0$, equation \eqref{lk} follows from Theorem \ref{secondo} once we observe that there exists a unique solution of \eqref{lk} in the interval $(0,\rho_0)$.
\smallskip

We are now interested in the behaviour of the process taking first the time and after the number of particles to infinity. We will do it in the mean field regime, as the general case is more complicated due to the spatial dependence in the infection rate. Using the monotonicity of the stochastic densities of healthy and recovered individuals it is possible to prove that, when $\gamma$ is fixed, $x^\gamma(t)$ and $y^\gamma(t)$ converge almost surely, as the time increases, to a random value $x^\gamma(\infty)$ and $0$ respectively.

The following results are  focused on the analysis of the final value of the survivors fraction $x^\gamma(\infty)$ when $\gamma$ approaches $0$. 

As it is shown by Theorem \ref{051} below, $\beta=1$ is a critical value for the system, because its behaviour significantly changes if the infection parameter $\beta$ is below or above $1$. The case $\beta<1$ is totally analysed in Theorem \ref{051}. The case $\beta>1$ is more delicate when the system starts with an infinitesimal density of infected and requires an intermediate step.
\vspace{+10pt}
Theorem \ref{051} shows that, when $\beta>1$ and hypothesis 
\vspace{-5pt}
\begin{align}
\mathscr{H}:\; x^\gamma\pare{0}=1-\gamma^{\alpha}, y^\gamma\pare{0}=\gamma^{\alpha}, \alpha\in\pare{0,\frac{d}{2}}
\end{align}
\vspace{-5pt}
is satisfied, $x^\gamma\pare{\infty}$ approaches, as $\gamma\to 0$, a value $\hat x_\infty:=\hat x_\infty\pare{\beta}$ which is the smaller of solution of 
\begin{align}\label{plmj}
x=e^{\beta\pare{x-1}}.
\end{align}
Observe that, when $\beta>1$, equation \eqref{plmj} admits two solutions, $1$ and $\hat x_\infty$ which is smaller than $\frac{1}{\beta}$. This follows from the fact that the function $g\pare{x}:=x-e^{\beta\pare{x-1}}$ is such that $g\pare{0}<0$, $g\pare{1}=0$ and admits a unique maximum value in $1-\frac{1}{\beta}\log\beta<\frac{1}{\beta}$. 
\vspace{+10pt}

In the next theorem we show that, under the hypothesis $\mathscr{H}$ and $\beta>1$, the density of susceptibles reaches every value in the interval $\dis{\pare{\hat x_\infty, 1}}$ at a random time at which it is possible to control also the fraction of infected. In other words a small perturbation of the equilibrium point $\pare{1,0}$ makes the system escape from it.

Before stating the theorem we need to observe that $x^\gamma\pare{t}$ and $y^\gamma\pare{t}$ take values in the set $\llav{\gamma, 2\gamma, \dots, 1}$, for this reason, for every $x\in [0,1]$, we define $\dis{[x]_\gamma:=\max_{n=1, \dots, \gamma^{-1}}\llav{\gamma n : \gamma n\leq x}}$ and the stopping time  
\begin{align}
S_{\gamma, x}:=\inf \llav{t\geq 0: x^\gamma\pare{t}=[x]_\gamma},
\end{align}
with the usual convention that the infimum of the empty set is $+\infty$. The following Theorem holds.
\begin{theorem}\label{balu}
Suppose  $\mathscr H$ is satisfied and $\beta>1$. Let $x^*\in\pare{\hat x_\infty, 1}$, then
\begin{align}\label{nahg}
\lim_{\gamma\to 0}\BB P\pare{S_{\gamma, x^*}<+\infty}=1.
\end{align}
Moreover, calling $y^*:=y^*\pare{x^*}=1-x^*+\frac{1}{\beta}\log\pare{x^*}$, we have that, 
\begin{align}\label{nahgg}
\lim_{\gamma\to 0}\mathbb P\pare{\abs{ y^\gamma\pare{S_{\gamma, x^*}}-y^*}> \varepsilon}=0,
\end{align}
for every $\varepsilon>0$.
\end{theorem}

The following result gives an estimate of the stopping time $S_{\gamma, x^*}$. We prove that it is of order $\log \gamma^{-1}$ and it is contained in a small deterministic interval, with probability converging to $1$.

\begin{theorem}\label{004}
Suppose $\mathscr H$ holds and  $\beta>1$. For all $x^*\in (\hat x_\infty, 1)$ there exists $T_{x^*}>0$ such that, for every $T>T_{x^*}$,
\begin{align}\label{jaja}
\lim_{\gamma\to 0}\BB P\pare{S_{\gamma, x^*}\in [t_c-T+\CAL L_{T, x^*}^{\text{min}}, t_c-T+\CAL L_{T, x^*}^{\text{max}}]}=1,
\end{align}
where $t_c=\frac{\alpha}{\beta-1}\log\gamma^{-1}$, while $\CAL L_{T, x^*}^{\text{min}}$ and $\CAL L_{T, x^*}^{\text{max}}$ are positive  deterministic times, depending on $T$ and $x^*$, such that 
\begin{align}\label{polka}
\dis{\lim_{T\to +\infty}\CAL L_{T, x^*}^{\text{max}}-\CAL L_{T, x^*}^{\text{min}}=0}.
\end{align}
\end{theorem}
In the next theorem we finally prove that, when $\beta< 1$ or when the system starts with a non infinitesimal density of infected, the limiting behaviour of the system, taking first the time and after the number of particles to infinity, coincides with the long time behaviour of the hydrodynamic limit, in other words it is possible to exchange the order of the limits in $t$ and $\gamma$. 
Things are different when $\beta>1$, indeed, under condition $\mathscr{H}$, the infection spreads all over the space and $x^\gamma (\infty)$ approaches a limit value $\hat x_\infty$ much smaller than $1$, so in this case the result significantly differs from the one predicted by the hydrodynamic limit.
\begin{theorem}\label{051}
Suppose that hypothesis of corollary \ref{teoo} hold. 
\begin{itemize}
\item[(a)]  If 
\begin{align}
\dis{(1)\; \rho_1=0 \text{  and  } \beta<1\qquad \text{ or }\qquad (2)\;\rho_1\in (0,1)},
\end{align}
then, for all $\varepsilon>0$, 
\vspace{-10pt}
\begin{align}
\lim_{\gamma\to 0}\BB P\pare{\abs{x^{\gamma}(\infty)-x_\infty}>\varepsilon}=0,
\end{align}
where $x_\infty$ is defined as in \eqref{lop}.
\item[(b)] If $\beta>1$ and $\mathscr{H}$ holds, then for all $\varepsilon>0$
\begin{align}
\lim_{\gamma\to 0}\BB P\pare{\abs{x^{\gamma}(\infty)-\hat x_\infty}>\varepsilon}=0,
\end{align}
where $\hat x_\infty$ is the smaller of solution of \eqref{plmj}. 
\end{itemize}
\end{theorem}

Despite it is a known result in literature (see \cite{UZS13} for instance), in the following remarks, we briefly analyze the qualitative behaviour of  the solution of the hydrodynamic limit in the mean field regime; some results will be useful in the sequel.
\begin{remarker}\label{rem1}
Let $\pare{x\pare{t}, y\pare{t}}_{t\geq 0}$ be the solution of \eqref{222} and consider the non trivial case in which $\rho_1\in(0,1)$ and $\rho_0\neq 0$.  By Theorem \ref{secondo} we know that $x\pare{t}$ is decreasing and, for all $t\geq 0$, $\rho_0e^{-\beta\pare{\rho_0+\rho_1}}\leq x\pare{t}\leq \rho_0$ while $0\leq y\pare{t}\leq 1-x\pare{t}$. Since 
\begin{displaymath}
\frac{dy}{dx}=-1+\frac{1}{\beta x},
\end{displaymath}
we have that
\begin{equation}\label{445}
y(x)=-x+\frac{1}{\beta}\log x+\rho_0+\rho_1-\frac{1}{\beta}\log\rho_0.
\end{equation}
To analyze the behaviour of $y(t)$ we need to distinguish two cases. If $\rho_0\leq\frac{1}{\beta}$,
$$\frac{d}{dt}y(t)\bigg|_{t=0}=(\beta\rho_0-1)\rho_1\leq 0.$$
As $x(t)\leq\rho_0$, for $t\geq 0$, we can conclude that $y(t)$ is decreasing. If  $\rho_0>\frac{1}{\beta}$, studying the sign of derivative of the function $y\pare{t}$ and using that $x\pare{t}$ is decreasing, it is possible to prove the existence of a time $t^\text{M}>0$ up to which $y(t)$ increases and after decreases. As $t^\text{M}$ is such that  $x(t^\text{M})=\frac{1}{\beta}$,  by \eqref{445} we get that the maximum value achievable by the function $y$ is
\begin{displaymath}
\begin{split}
y(t^\text{M})=\rho_0+\rho_1-\frac{1}{\beta}-\frac{1}{\beta}\log(\beta\rho_0).
\end{split}
\end{displaymath}
By Corollary \ref{teoo} $\big(x(t), y(t)\big)$ converges to the equilibrium point $\pare{x_\infty, 0}$, where $x_{\infty}$ is defined in \eqref{lk} and is strictly less than $\min\pare{\frac{1}{\beta}, \rho_0}$.
\end{remarker}
\begin{remarker}\label{bgf}
Observe that, from \eqref{lk}, it is possible to analyze the relation between the parameters of the model in the particular case in which $\rho_1=1-\rho_0$. In the sequel remember that $x_\infty$ is always contained in the interval $\pare{0, \min\pare{\rho_0, \frac{1}{\beta}}}$. The evolution of $\rho_0$ respect to $x_\infty$, when $\beta$ fixed, is given by
\begin{displaymath}
\rho_0(x_\infty)=x_\infty e^{\beta\pare{1- x_\infty}}.
\end{displaymath}
Calling $x_\infty^\text{M}$ the first positive solution of \eqref{6kiu} in  the variable $x$
\begin{align}\label{6kiu}
1=xe^{\beta\pare{1-x}},
\end{align}
we have that the function $\rho_0\pare{x_\infty}$ is defined for $x_\infty\in \pare{0, x_\infty^\text{M}}$ and that $x_\infty^\text{M}=1$ if $\beta\leq1$, while $x_\infty^\text{M}<1$ if $\beta> 1$.
The function $\rho_0(x_\infty)$ increases from $0$ to $1$.
The relation between $\beta$ and $x_\infty$ when $\rho_0$ is fixed, is 
$$\beta(x_\infty)=\frac{\log x_\infty-\log \rho_0}{x_\infty-1}.$$
The function $\beta(x_\infty)$ is defined for $x_\infty\in \pare{0, \rho_0}$ and decreases from $+\infty$ to $0$.
Finally the relation between $\rho_0$ and $\beta$ when $x_\infty$ is fixed is
$$\rho_0(\beta)=x_\infty e^{ \beta\pare{1-x_\infty}},$$
the function is well defined for $\beta\in\pare{0, \frac{\log x_\infty}{x_\infty-1}}$ and increases from $x_\infty$ to $1$. All the previous functions are  invertible, so it is possible to deduce also the behaviour of their inverses. 
\end{remarker}

The rest of the paper is organized as follows.
In section \ref{sez2} we prove the hydrodynamic limit result stated by Theorem \ref{teo}. Section \ref{sez3} is devoted to the proof of Theorem \ref{secondo}. Finally, in section \ref{cinque} we prove the result about the long time behaviour of the system in the mean field regime introduced by Theorems \ref{balu}, \ref{004} and \ref{051}.

\section{Proof of Theorem \ref{teo}}\label{sez2}
Let $\CAL M_1^+$ be the space of positive measures on $\BB T^d$ with mass bounded by $1$, endowed with the weak topology and the distance $d_{\CAL M_1^+}$ defined in the Appendix.  Fix $T>0$  and define $\mathcal{D} := D\pare{[0,T], \mathcal{M}_1^+}$, the space of right continuous functions with left limits taking values in $\mathcal M_1^+$; we endow $\CAL D$ with the modified Skorohod metric $d_{\text{SK}}$ (see the Appendix for details). From now on we use the capital letter $\Pi$ to denote a process in $\mathcal D$ and the underlined letter $\underline\Pi$ to denote a vector in $\mathcal D^2$.                Call $\underline\Pi^\gamma=\pare{\Pi^{\gamma, 0}, \Pi^{\gamma, 1}}=\pare{\pi_t^{\gamma, 0}, \pi_t^{\gamma, 1}}_{t\in[0,T]}$ the process introduced in \eqref{0}. Let $i\in\llav{0,1}$, since there are jumps the process $\Pi^{\gamma, i}\in\mathcal D$.
The process $\underline\Pi^\gamma$ takes values in the space $\CAL D^2$, which is endowed with the product metric $d$ (see the Appendix for details).

Consider the sequence of probability measures $\llav{P^\gamma}_{\gamma>0}$ on $\CAL{D}^2$ corresponding to the Markov process $\underline\Pi^\gamma$. Following the same steps of \cite{KL}, chapter 4, the proof consists in showing the convergence of $\llav{P^\gamma}_\gamma$ to the Dirac measure concentrated on a deterministic path $\underline\Pi^*=\pare{\pi_t^{*,0},\pi_t^{*,1}}_{t\in [0,T]}$, where $\pi_t^{*,0}$ and $\pi_t^{*,1}$ are measures on $\BB T^d$ absolutely continuous with respect to the Lebesgue measure and their densities $u_0(t,r)$, $u_1(t,r)$ satisfy \eqref{2}. After we will argue that the convergence in distribution to a deterministic trajectory in $\CAL C^2$ (where $\CAL C:=\CAL C\pare{[0,T], \CAL M_1^+}$ denotes the set of all the continuous functions defined in $[0,T]$ and taking values in $\CAL M_1^+$) implies uniform convergence in probability in the sense stated by Theorem \ref{teo}.

In order to prove the convergence of the sequence $\llav{P^\gamma}_\gamma$ we proceed in two main steps, first we prove the tightness of the sequence  and then we show that all converging subsequences converge to the same limit.\\
Let $G\in C\pare{\BB T^d, \BB R}$  and $i\in\llav{0,1}$, by Lemma A1.5.1 of \cite{KL} we know that
\begin{align}\label{33}
M_t^{\gamma, i, G}:=\pic{\pi_t^{\gamma, i},G}-\pic{\pi_0^{\gamma,i},G}-\int_0^tL_\gamma\pic{\pi_s^{\gamma,i},G}ds
\end{align}
and
\begin{align}\label{4567}
N_t^{\gamma, i, G}:=\pare{M_t^{\gamma, i, G}}^2-\int_0^tL_\gamma\pic{\pi_s^{\gamma,i},G}^2-2\pic{\pi_s^{\gamma,i},G}L_\gamma\pic{\pi_s^{\gamma, i},G}ds
\end{align}
are martingales with respect to the natural filtration generated by the process. Computations show that
\begin{align}\nonumber
L_\gamma\pic{\pi_s^{\gamma,0},G}=-\beta\gamma^d\sum_{x\in \BB T_\gamma^d}\mathbb{I}_{\llav{\eta^\gamma_s(x)=0}}\gamma^d\sum_{y\in \BB T_\gamma^d}\mathbb{I}_{\llav{\eta^\gamma_s(y)=1}}J\pare{\gamma x, \gamma y}G(\gamma x).
\end{align}
Observe that
\begin{align}\nonumber
\gamma^d\sum_{y\in \BB T_\gamma^d}\mathbb{I}_{\{\eta^\gamma_s(y)=1\}}J\pare{\gamma x, \gamma y}=\pic{\pi_s^{\gamma,1}J\pare{\gamma x,\cdot}},
\end{align}
then \eqref{33} becomes 
\begin{align}\label{5}
M_t^{\gamma, 0, G}=\pic{\pi_t^{\gamma,0},G}-\pic{\pi_0^{\gamma,0},G}+\int_0^t\beta\pic{\pi_s^{\gamma,0}\;,\; \pic{\pi_s^{\gamma,1}, J\pare{\gamma x,\cdot}} G}ds.
\end{align}
In a similar way we get
\begin{align}\label{6}
M_t^{\gamma,1,G}=\pic{\pi_t^{\gamma,1},G}-\pic{\pi_0^{\gamma,1},G}-\int_0^t\beta\pic{\pi_s^{\gamma,0}, \;\pic{\pi_s^{\gamma,1}, J\pare{\gamma x,\cdot}} G}-\pic{\pi_s^{\gamma,1}, G}ds.
\end{align}
The following lemma holds.
\begin{lemma}\label{ray}
For every $G\in C\pare{\BB T^d, \BB R}$, $i\in\llav{0,1}$ and $\varepsilon>0$,
\begin{align}
P\pare{\sup_{t\in [0,T]}\abs{M_t^{\gamma, i, G}}>\varepsilon}\leq\varepsilon^{-2}C\gamma^d T,
\end{align}
where $C$ is a constant which depends on $\beta$, $\norm{G}_\infty$ and $\norm{J}_\infty$.
\end{lemma}
\begin{proof}
We will prove the lemma for $i=0$ as the proof for $i=1$ is similar. 
By \eqref{4567} and Doob's inequality we get that
\begin{align}\label{marc}
\begin{aligned}
P\pare{\sup_{t\in [0, T]}\abs{M_t^{\gamma, 0, G}}>\varepsilon}&\leq\varepsilon^{-2}\BB  E_P\pare{M_T^{\gamma, 0, G}}^2\\
&=\varepsilon^{-2}\BB E_P\pare{\int_0^TL_\gamma\pic{\pi_s^{\gamma,0},G}^2-2\pic{\pi_s^{\gamma,0},G}L_\gamma\pic{\pi_s^{\gamma, 0},G}ds}.
\end{aligned}
\end{align}
 Computations shows that, for every $t\geq 0$,
\begin{align}\label{rana}
\begin{aligned}
L_\gamma&\pic{\pi_t^{\gamma,0},G}^2-2\pic{\pi_t^{\gamma,0},G}L_\gamma\pic{\pi_t^{\gamma, 0},G}=\\
&=\beta\gamma^{3d}\sum_{x\in \BB T_\gamma^d}\sum_{y\in \BB T_\gamma^d}J\pare{\gamma x, \gamma y}\mathbb{I}_{\llav{\eta^\gamma_t(y)=1}}\mathbb{I}_{\llav{\eta^\gamma_t(x)=0}}G(\gamma x)^2.
\end{aligned}
\end{align}
This implies that the right hand side of \eqref{marc} can be bounded by $\varepsilon^{-2}C\gamma^d T$, where $C=\beta\norm{G}_\infty^2\norm{J}_\infty$. This concludes the proof.
\end{proof}
The tightness of the sequence $P^\gamma$ is guaranteed by Theorem \ref{tightness}; as the proof is standard it is postponed in the Appendix.
Theorem \ref{tightness} implies that any subsequence of $P^\gamma$ has a convergent sub-subsequence; it remains then to characterize all the limit points of the sequence $P^\gamma$. Let $P^*$ be a limit point and $P^{\gamma_n}$ be a subsequence converging to $P^*$. The following Lemma holds.
\begin{lemma}\label{rrrrr}
It holds that
\begin{align}\nonumber
P^*\pare{\CAL C\pare{[0,T], \mathcal{M}_1^+}^2}=1
\end{align}
\end{lemma}
\begin{proof}
Let $P^{*,i}$ and $P^{\gamma_n, i}$ be the $i$-th marginal of $P^*$ and $P^{\gamma_n}$ respectively. Observe that $P^{\gamma_n, i}$ converges to $P^{*,i}$ as $\gamma_n$ goes to $0$. The proof follows once we show that $P^{*, i}$  is concentrated on continuous trajectories  for $i\in\llav{0,1}$. We will prove it just for $i=0$, the case $i=1$ is analogous. As the function $\Delta$,  defined on the elements $\Pi=\pare{\pi_t}_{t\in[0,T]}\in\CAL D$ as
\begin{align}\nonumber
\Delta\pare{\Pi}=\sup_{t\in [0,T]}d_{\CAL M_1^+}\pare{\pi_t, \pi_{t-}},
\end{align}
is continuous, by weak convergence it is enough to show that 
\begin{align}\label{liuy}
\lim_{\gamma_n\to 0}\BB E_{P^{\gamma_n,0}}\pare{\Delta}=0.
\end{align}
For every $t\in[0,T]$ and for every $f_k$ belonging to the sequence introduced in the Appendix to define $d_{\CAL M_1^+}$, the following holds $P$-a.s.
\begin{align}\label{hhggffll}
\abs{\pic{\pi_t^{\gamma,0},f_k}-\pic{\pi_{t_-}^{\gamma, 0},f_k}}&=\abs{\gamma^d\sum_{x\in T_\gamma^d} \pare{\mathbb{I}_{\{\eta^\gamma_t(x)=0\}}-\mathbb{I}_{\{\eta^\gamma_{t_-}(x)=0\}}}f_k(\gamma x)}\leq\gamma^d \|f_k\|_{\infty},
\end{align}
The last inequality follows from the fact that in the particle system there is only one jump at a time with probability $1$ and consequently $\dis{\sum_{x\in T_\gamma^d} \abs{\mathbb{I}_{\{\eta^\gamma_t(x)=0\}}-\mathbb{I}_{\{\eta^\gamma_{t_-}(x)=0\}}}\leq 1}$.
By \eqref{hhggffll} we get 
\begin{align}\label{voi}
\mathbb{E}_{P^{\gamma_n}}\pare{\Delta}&=\mathbb{E}_{P}\pare{\sup_{t\in [0,T]}\sum_{k=1}^{+\infty}\frac{1}{2^k}\frac{\abs{\pic{\pi^{\gamma_n, i}_t,f_k}-\pic{\pi^{\gamma_n, i}_{t_-},f_k}}}{1+\abs{\pic{\pi^{\gamma_n, i}_t,f_k}-\pic{\pi^{\gamma_n, i}_{t_-},f_k}}}} \\
&\leq\sum_{k=1}^{N}\frac{1}{2^k}\gamma_n^d \|f_k\|_{\infty}+\sum_{k=N+1}^{+\infty}\frac{1}{2^k},
\end{align}
where $N$ is a fixed integer. Letting first $\gamma_n$ to $0$ and after $N$ to $+\infty$ in the right hand side of \eqref{voi} we get \eqref{liuy}.
\end{proof}
Let  $\Psi_0, \Psi_1:\CAL D^2\to\BB R$ be the functions defined on the elements $\underline\Pi=\pare{\Pi^0,\Pi^1}=\pare{\pi^0_t, \pi^1_t}_{t\in [0,T]}$ as 
\begin{align}\label{luno}
&\Psi_0\pare{\underline\Pi}=\sup_{t\in[0,T]}\abs{\pic{\pi_t^{  0},G}-\pic{\pi_0^{  0},G}+\int_0^t\beta\pic{\pi_s^{  0}\;,\; \pic{\pi_s^{  1}, J\pare{r,\cdot}} G}ds}
\end{align}
and
\begin{align}\label{lunos}
&\Psi_1\pare{\underline\Pi}=\sup_{t\in[0,T]}\abs{\pic{\pi_t^{  1},G}-\pic{\pi_0^{  1},G}+\int_0^t\beta\pic{\pi_s^{  0}\;,\; \pic{\pi_s^{  1}, J\pare{r,\cdot}} G}-\pic{\pi_s^{  1},G}ds}.
\end{align}
Given $i\in\llav{0,1}$ call $D_{\Psi_i}$ the set of the discontinuity points of $\Psi_i$ and $D_{\Psi_i}^c$ its complementar set.
By Lemma \ref{ray}, \eqref{5} and \eqref{6} we know that 
\begin{align}
\BB P^{\gamma_n}\pare{\abs{\Psi_i\pare{\underline \Pi}}>\varepsilon}\xrightarrow[\gamma_n\to 0]{}0,
\end{align}
for all $i\in\llav{0,1}$ and for all $\varepsilon>0$. In Lemma \eqref{jooe} below we prove that the previous limit coincides with $\dis{\BB P^{*}\pare{\abs{\Psi_i\pare{\underline \Pi}}>\varepsilon}}$ which is consequently equal to $0$. In order to do that, by the Continuous Mapping Theorem (see \cite{Bi}, Theorem 2.7), we need to show that $\BB P^*\pare{D_{\Psi_i}}=0$. In the folowing lemma we prove that $\CAL C^2\subseteq D_{\Psi_i}^c$, for $i\in\llav{0,1}$.

\begin{lemma}\label{lopi}
The functions $\Psi_0$ and $\Psi_1$ defined in \eqref{luno}-\eqref{lunos} are continuous at the elements of $\CAL C^2$.
\end{lemma}
\begin {proof}
We will prove the lemma just for the function $\Psi_0$ as the proof for $\Psi_1$ is similar. Consider a sequence $\underline\Pi^n\in \CAL D^2$ and $\underline\Pi\in \CAL C^2$  such that $d\pare{\underline\Pi^n, \underline\Pi}\xrightarrow[n\to\infty]{}0$. This implies that $d_{\text{SK}}\pare{\Pi^{n,i}, \Pi^{i}}\xrightarrow[n\to\infty]{}0$ for all $i\in\llav{0,1}$. The proof follows once we show that $\Psi_0\pare{\underline\Pi^n}\xrightarrow[n\to\infty]{}\Psi_0\pare{\underline\Pi}$.\\
Since $d_{\text{SK}}\pare{\Pi^{n,i}, \Pi^i}\xrightarrow[n\to\infty]{}0$ and  $\Pi^i\in \CAL C$ for all $i\in\llav{0,1}$, then $\dis{\sup_{t\in [0,T]}d_{\CAL M_1^+}\pare{\pi_t^{n,i}, \pi_t^i}}\xrightarrow[n\to +\infty]{}0$ (convergence in the Skorohod metric implies uniform convergence) and consequently
\begin{align}\label{hu}
\sup_{t\in [0,T]}\abs{\pic{\pi_t^{n,i}, G}-\pic{\pi_t^i, G}}\xrightarrow[n\to\infty]{}0,
\end{align}
for $i\in\llav{0,1}$ and $G\in C\pare{\BB T^d, \BB R}$.
\begin{align}\label{pok}
\begin{aligned}
\abs{\Psi_0\pare{\underline\Pi^n}-\Psi_0\pare{\underline\Pi}}&\leq 2\sup_{t\in [0,T]}\abs{\pic{\pi_t^{n,0}, G}-\pic{\pi_t^0, G}}+\\
&+\beta\int_0^Tds\abs{\pic{\pi_s^{ n, 0}\;,\; \pic{\pi_s^{ n, 1}, J\pare{r,\cdot}} G}-\pic{\pi_s^{  0}\;,\; \pic{\pi_s^{  1}, J\pare{r,\cdot}} G}}.
\end{aligned}
\end{align}
The first term in the right hand side of \eqref{pok} vanishes because of \eqref{hu}. To show that the integral converges to $0$ we can apply the dominated convergence theorem, thus, to conclude, it is enough to show that the integrand term appearing in \eqref{pok} vanishes for all $s\in [0,T]$. Fix $s$ and observe that 
\begin{align}
\begin{aligned}
&\abs{\pic{\pi_s^{ n, 0}\;,\; \pic{\pi_s^{ n, 1}, J\pare{r,\cdot}} G}-\pic{\pi_s^{  0}\;,\; \pic{\pi_s^{  1}, J\pare{r,\cdot}} G}}\leq\\ 
&\abs{\pic{\pi_s^{ n, 0}\;,\; \pic{\pi_s^{ n, 1}, J\pare{r,\cdot}} G}-\pic{\pi_s^{ n, 0}\;,\; \pic{\pi_s^{  1}, J\pare{r,\cdot}} G}}\\
&+\abs{\pic{\pi_s^{ n, 0}\;,\; \pic{\pi_s^{ 1}, J\pare{r,\cdot}} G}-\pic{\pi_s^{  0}\;,\; \pic{\pi_s^{  1}, J\pare{r,\cdot}} G}},
\end{aligned}
\end{align}
$\abs{\pic{\pi_s^{ n, 0}\;,\; \pic{\pi_s^{ 1}, J\pare{r,\cdot}} G}-\pic{\pi_s^{  0}\;,\; \pic{\pi_s^{  1}, J\pare{r,\cdot}} G}}$ vanishes because of \eqref{hu}. To conclude we just need to show that
\begin{align}
\begin{aligned}
&\abs{\pic{\pi_s^{ n, 0}\;,\; \pic{\pi_s^{ n, 1}, J\pare{r,\cdot}} G}-\pic{\pi_s^{ n, 0}\;,\; \pic{\pi_s^{  1}, J\pare{r,\cdot}} G}}\xrightarrow[n\to\infty]{}0
\end{aligned}
\end{align}
Divide the torus in $\varepsilon^{-d}$ $d$-dimensional cubes $\llav{I_j}_{j=1}^{\varepsilon^{-d}}$ whose side has length $\varepsilon$ and call $r_j$ the center of the $j$-th interval, we get
\begin{align}\label{zxc}
\begin{aligned}
&\abs{\pic{\pi_s^{ n, 0}\;,\; \pic{\pi_s^{ n, 1}, J\pare{r,\cdot}} G}-\pic{\pi_s^{ n, 0}\;,\; \pic{\pi_s^{  1}, J\pare{r,\cdot}} G}}\\
&\leq \norm{G}_\infty\sum_{j=1}^{\varepsilon^{-d}}\int_{I_j}\pi_s^{n,0}(dr)\abs{\pic{\pi_s^{ n, 1}, J\pare{r,\cdot}}-\pic{\pi_s^{  1}, J\pare{r,\cdot}}}\\
&\leq\norm{G}_\infty\sum_{j=1}^{\varepsilon^{-d}}\int_{I_j}\pi_s^{n,0}(dr)\Big(\abs{\pic{\pi_s^{ n, 1}, J\pare{r,\cdot}}-\pic{\pi_s^{ n, 1}, J\pare{r_j,\cdot}}}+\abs{\pic{\pi_s^{ n, 1}, J\pare{r_j,\cdot}}-\pic{\pi_s^{  1}, J\pare{r_j,\cdot}}}+\\
&\qquad\qquad\qquad\qquad\qquad+\abs{\pic{\pi_s^{  1}, J\pare{r_j,\cdot}}-\pic{\pi_s^{  1}, J\pare{r,\cdot}}}\Big)\\
&\leq C\pare{ \varepsilon+\sum_{j=1}^{\varepsilon^{-d}}\int_{I_j}\pi_s^{n,0}(dr)\abs{\pic{\pi_s^{ n, 1}, J\pare{r_j,\cdot}}-\pic{\pi_s^{  1}, J\pare{r_j,\cdot}}}}\\
&\leq  C \pare{\varepsilon+\abs{\pic{\pi_s^{ n, 1}, J\pare{r_j,\cdot}}-\pic{\pi_s^{  1}, J\pare{r_j,\cdot}}}},
\end{aligned}
\end{align}
where $C$ is a constant which depends on $d$, $\|G\|_\infty$ and $\|J'\|_\infty$. The right hand side of \eqref{zxc} converges to $0$ letting first $\varepsilon$ to $0$ and after $n$ to $\infty$. This concludes the proof.
\end{proof}
\begin{lemma}\label{jooe}
For every $G\in C\pare{\BB T^d, \BB R}$, $P^*$ is concentrated on trajectories $\underline\Pi=\pare{\pi_t^0, \pi_t^1}_{t\in[0,T]}$ such that 
\begin{align}\label{452}
&\pic{\pi_t^{  0},G}=\pic{\pi_0^{  0},G}-\int_0^t\beta\pic{\pi_s^{  0}\;,\; \pic{\pi_s^{  1}, J\pare{r,\cdot}} G}ds,\\
&\pic{\pi_t^{  1},G}=\pic{\pi_0^{  1},G}+\int_0^t\beta\pic{\pi_s^{  0}, \;\pic{\pi_s^{  1}, J\pare{r,\cdot}} G}-\pic{\pi_s^{  1}, G}ds,\label{3217}
\end{align}
for every $t\in [0,T]$.
\end{lemma}
\begin{proof}
We will prove just relation \eqref{452} as to prove \eqref{3217} we proceed in a similar way.  Let $\epsilon>0$. By Lemma \ref{ray}
\begin{equation}\label{a}
P\pare{\sup_{t\in [0,T]} \abs{M_t^{\gamma, 0,G}}>\epsilon}\xrightarrow[\gamma\to 0]{}0.
\end{equation}
Since the map $\Psi_0$ defined in lemma \ref{lopi} is continuous at the elements of  $\CAL C^2$, using the Continuous Map Theorem (see Theorem 2.7 of \cite{Bi}) and Lemma \ref{rrrrr}  we can deduce that
\begin{displaymath}
\begin{split}
&P^*\pare{\sup_{t\in[0,T]}\abs{\pic{\pi_t^{  0},G}-\pic{\pi_0^{  0},G}+\int_0^t\beta\pic{\pi_s^{  0}\;,\; \pic{\pi_s^{  1}, J\pare{r,\cdot}} G}ds}>\epsilon}\\
&=\lim_{\gamma_n\to 0}P^{\gamma_n}\pare{\sup_{t\in[0,T]}\abs{\pic{\pi_t^{  0},G}-\pic{\pi_0^{  0},G}+\int_0^t\beta\pic{\pi_s^{  0}\;,\; \pic{\pi_s^{  1}, J\pare{r,\cdot}} G}ds}>\epsilon}=0.
\end{split}
\end{displaymath}
The last equality follows from \eqref{5} and \eqref{a}.
\end{proof}
We now prove that $P^*$ is concentrated on trajectories absolutely continuous with respect to the Lebesgue measure. To conclude it is enough to show that, for every $G\in C\pare{\BB T^d, \BB R}$ and $i\in\llav{0,1}$, the following holds
\begin{align}\label{40}
P^*\pare{\sup_{t\in[0,T]}\abs{\pic{\pi_t^{\gamma, i}, G}}\leq \int_{\BB T^d}\abs{G(r)}dr}=1.
\end{align}
Fix $G\in C\pare{\BB T^d, \BB R}$, $i\in\llav{0,1}$  and observe that
\begin{equation}\label{ziaa}
P\pare{\sup_{t\in[0,T]}\abs{\pic{\pi_t^{\gamma, i}, G}}\leq\gamma^d\sum_{x\in \BB T_\gamma^d}\abs{G(\gamma x)}}=1.
\end{equation}
Since the function which associates to a trajectory $\underline\Pi=\llav{\pi_t^0,\pi_t^1}\in\CAL D^2$ the quantity $\sup_{t\in [0,T]}\abs{\pic{\pi_t^i,G}}$ is continuous, by the weak convergence and \eqref{ziaa} we obtain that, for all $\varepsilon>0$,
\begin{align}\nonumber
\begin{aligned}
P^*&\pare{\sup_{t\in [0,T]}\abs{\pic{\pi_t^i, G}}-\int_{\BB T^d}\abs{G(r)}dr>\varepsilon}\\
&\leq\liminf_{\gamma_n\to 0} P^{\gamma_n}\pare{\sup_{t\in[0,T]}\abs{\pic{\pi_t^{ i}, G}}-\int_{\BB T^d}\abs{G(r)}dr>\varepsilon}=0,
\end{aligned}
\end{align}
and  \eqref{40} is proved. In particular $P^*$ is concentrated on trajectories $\underline\Pi=\pare{\pi_t^0, \pi_t^1}_{t\in [0,T]}$ whose densities at time $0$ are $\rho_0(r)$ and $\rho_1(r)$ respectively, indeed, for every $\epsilon>0$,
\begin{align}\nonumber
\begin{aligned}
&P^*\pare{\underline\Pi:\abs{\pic{\pi_0^0,G}-\int_{\mathbb{T}^d}G(r)\rho_0(r)dr}>\epsilon}\\
&\leq\liminf_{\gamma_n\to 0}P^{\gamma_n}\Big(\underline\Pi:\abs{\pic{\pi_0^0,G}-\int_{\mathbb{T}^d}G(r)\rho_0(r)dr}>\epsilon\Big)\\
&=\liminf_{\gamma_n\to 0}P\pare{\eta^{\gamma_n}:\abs{\gamma_n^d\sum_{x\in T_{\gamma_n}}\mathbb{I}_{\{\eta^{\gamma_n}(x)=0\}}G(\gamma_n x)-\int_{\mathbb{T}^d}G(r)\rho_0(r)dr}>\epsilon}=0.
\end{aligned}
\end{align}
We proceed in a similar way to prove that 
\begin{align}
P^*\pare{\underline\Pi:\abs{\pic{\pi_0^1,G}-\int_{\mathbb{T}^d}G\pare{r}\rho_1\pare{r}dr}>\varepsilon}=0.
\end{align}
The  previous results show that every limit point $P^*$ is concentrated on absolutely continuous trajectories $\llav{\pi_t^0(dr), \pi_t^1(dr)}_{t\in[0,T]}=\llav{ u_0(t,r)dr, u_1(t,r)dr}_{t\in[0,T]}$  whose densities are weak solutions of the system \eqref{452}-\eqref{3217} and at time $0$ are equal to $\rho_0(r)$ and $\rho_1(r)$ respectively. To prove the convergence of the entire sequence $P^\gamma$ it remains to show the uniqueness of the limit points, and this follows from the existence and uniqueness of the solution of \eqref{2} which is guaranteed by Theorem \ref{secondo}.  
To conclude, the sequence $P^\gamma$ converges to the Dirac measure centrated on the trajectory $\underline{\CAL U}=\llav{ u_0(t,r)dr, u_1(t,r)dr}_{t\in[0,T]}$ where $\pare{u_0(t, r), u_1(t, r)}$  is the unique solution of \eqref{2}. Thus $\underline\Pi^\gamma$ converges in distribution, as $\gamma\to 0$, to the deterministic trajectory $\underline{\CAL U}$. Since convergence in distribution to a deterministic variable implies convergence in probability we get that
\begin{equation}\label{er}
d\pare{\underline{\Pi}^\gamma, \underline{\CAL U}}\xrightarrow[\gamma\to 0]{P}0.
\end{equation}
By \eqref{er} and the fact that $\underline{\CAL U}\in \CAL C^2$ we get \eqref{anna}.

\section{Proof of Theorem \ref{secondo}}\label{sez3}
\paragraph*{Proof of existence, uniqueness and item (a).}

For $t\geq 0$ and $c>0$, call 
\begin{align}
\CAL C_{t, c}=\llav{f\in C\pare{[0, t]\times \BB T^d, \BB R}: \norm{f}_{t, \infty}\leq c},
\end{align}
where 
$\dis{\|f\|_{t,\infty}:=\max\llav{\abs{f(s,r)}, {s\in[0,t], r\in\BB T^d}}}$.
Let  $\CAL C_{t, c}\times\CAL C_{t, c}$ be endowed with the metric $d_\infty\pare{\pare{f,g}, (\tilde f, \tilde g)}:=\|f-\tilde f\|_{t,\infty}+\norm{g-\tilde g}_{t,\infty}$, we call $\Phi_0$ and $\Phi_1$ the maps defined on $\CAL C_{t, c}\times\CAL C_{t, c}$ as
\begin{align}\label{wast}
\begin{aligned}
&\Phi_0\pare{x(t,r), y(t,r)}=\rho_0(r)-\int_0^tds\int_{\mathbb{T}^d}dr'\beta J(r,r')y(s,r')x(s,r)\\
&\Phi_1\pare{x(t,r), y(t,r)}=\rho_1(r)+\int_0^tds\int_{\BB T^d} dr'\beta J(r,r') y(s,r') x(s,r)-y(s,r)
\end{aligned}
\end{align}
for every $x,y\in \CAL C_{t, c}$. It is possible to prove that for $\tilde t$ small enough, the map  $\pare{\Phi_0, \Phi_1}$ is a contraction in $\CAL C_{\tilde t, c}\times\CAL C_{\tilde t, c}$. By the Banach fixed-point Theorem (see \cite{Apo74}) we can deduce that there exists a unique solution of system \eqref{2}. 
It is sufficient to show that such a solution is always contained in the interval $[0,1]\times[0,1]$ and iterate the previous argument in order to show existence and uniqueness globally in $[0, +\infty)\times \BB T^d$.

We start proving that $u_i\pare{t, r}\geq 0$ for $i\in\llav{0,1}$. Observe that, for all $\pare{t, r}\in [0,+\infty)\times \BB T^d$,
\begin{align}\label{ssdw}
u_0(t,r)=\rho_0(r)e^{-\beta\int_0^{t}\int_{\mathbb{T}^d}J\pare{r,r'}u_1(s,  r')dr'ds},
\end{align}
while
\begin{align}\label{nb}
u_1(t,r)=e^{-t}\rho_1(r)+\beta \int_0^tds\;e^{-\pare{t-s}}\int_{\BB T^d}dr'J(r,r')u_1(s,r')u_0(s,r).
\end{align}
By \eqref{ssdw} it is obvious that $u_0\pare{t, r}\geq 0$ for all $\pare{t, r}\in [0, +\infty)\times \BB T^d$. To prove that the same property holds for the function $u_1\pare{\cdot, \cdot}$, define
\begin{align}
U\pare{t}:=\min\pare{0, \min_{r\in \BB T^d} u_1\pare{t, r}}, \quad t\in [0, +\infty),
\end{align}
which is non positive and such that $U\pare{0}=0$. Since
\begin{align}
U\pare{t}\geq \norm{J}_\infty\int_0^tU\pare{s}ds,
\end{align}
by Gronwall's inequality it follows that $U\pare{t}\geq U\pare{0}e^{\norm{J}_\infty t}\equiv 0$ and consequently $u_1\pare{t, r}\geq 0$ for all $\pare{t, r}\in [0, +\infty)\times \BB T^d$.

We define the function $V(t,r):=u_0(t,r)+u_1(t,r) $. To conclude the proof of item (a) it is enough to observe that $V(0,r)\leq 1$ and that $V(t,r)$ decreases in $t$ independently from the value of $\beta$ as $\frac{\partial}{\partial t}V(t,r)=-u_1(t,r)\leq 0$. 
\paragraph*{Proof of item (b).}
By item (a) it follows that $\frac{\partial}{\partial t}u_0\pare{t, r}\leq 0$, consequently $u_0\pare{t, r}$ decreases in $t$ and there exists $\dis{\rho_\infty\pare{r}:=\lim_{t\to +\infty} u_0\pare{t,r}}$ for every $r\in\BB T^d$. 

Then, by the monotonicity of $V\pare{t,r}$ in $t$, we get that there exist both $\dis{\lim_{t\to +\infty}V\pare{t,r}}$ and $\dis{\lim_{t\to +\infty}u_1\pare{t,r}}$. Since $\frac{\partial}{\partial t}V\pare{t, r}=-u_1\pare{t, r}$, we can conclude that $\dis{\lim_{t\to +\infty}\frac{\partial}{\partial t}V\pare{t,r}\equiv 0}$ which implies $\dis{\lim_{t\to +\infty}u_1\pare{t,r}\equiv 0}$. 
Since
 $$\frac{\partial}{\partial t}u_0= \beta u_0J*\frac{\partial}{\partial t}\pare{u_0+u_1},$$
then
$$\frac{\partial}{\partial t}\pare{\log u_0}=\beta\frac{\partial}{\partial t}\pare{ J*\pare{u_0+u_1}},$$
and consequently
\begin{align}\label{poq}
u_0\pare{t,r}=\rho_0\pare{r}e^{-\beta \int_{\BB T^d}dr'J\pare{r', r}\corch{\rho_0\pare{r'}+\rho_1\pare{r'}-\pare{u_0\pare{t, r'}+u_1\pare{t, r'}}}}.
\end{align}
\eqref{seii} follows taking the limit, as  $t\to +\infty$, in both sides of \eqref{poq}.

\section{Asymptotic behaviour of the microscopic model}\label{cinque}
In this section we will provide the proofs of Theorems \ref{balu}, \ref{004} and \ref{051}.
\subsection{Proof of Theorem \ref{balu}}
Using \eqref{5} and \eqref{6} with $G\equiv 1$, we get that
\begin{align}\label{lopxx}
M^{\gamma,x}\pare{t}:=x^\gamma\pare{t}-\pare{1-\gamma^\alpha}+\int_0^t\beta x^\gamma\pare{s}y^\gamma\pare{s}ds
\end{align}
and
\begin{align}\label{lopjk}
M^{\gamma,y}\pare{t}:=y^\gamma\pare{t}-\gamma^\alpha-\int_0^t\beta x^\gamma\pare{s}y^\gamma\pare{s}-y^\gamma\pare{s}ds 
\end{align}
are martingales with respect to the filtration $\pare{\mathscr{F}_\gamma\pare{t}}_{t\geq 0}$ generated by the process. \eqref{lopxx} and \eqref{lopjk} can be rewritten in differential form as
\begin{equation}\label{3}
\begin{cases}
dx^\gamma(t)=\pare{-\beta y^\gamma(t)+E^\gamma(t)}dt+dM^{\gamma, x}(t)\\
dy^\gamma(t)=\pare{\pare{\beta-1}y^\gamma(t)-E^\gamma(t)}dt+dM^{\gamma, y}(t)\\
x^\gamma(0)=1-\gamma^\alpha, y^\gamma(0)=\gamma^\alpha
\end{cases}
\end{equation}
where $E^\gamma(t)=\beta y^\gamma(t)\pare{1-x^\gamma(t)}$ is the non-linear error. By \eqref{3} and the Duhamel's formula we get that
\begin{align}
x^\gamma(t)&=1-\gamma^\alpha-\beta\int_0^t y^\gamma(s)ds+\int_0^tE^\gamma(s)ds+M^{\gamma, x}(t),\label{zaq6}\\ 
y^\gamma(t)&=\gamma^\alpha e^{\pare{\beta-1}t}-\int_0^te^{\pare{\beta-1}(t-s)}E^\gamma(s)ds+e^{\pare{\beta-1}t}\int_0^te^{-\pare{\beta-1}s}dM^{\gamma, y}(s).\label{zaq7}
\end{align}
From now on we denote by  $t_c:=\frac{\alpha}{\beta-1}\ln\gamma^{-1}$ the time at which the linear part of $y^\gamma\pare{t}$ (i.e. $\gamma^\alpha e^{\pare{\beta-1}t}$) is equal to $1$. By giving an estimate of the non-linear error $E^\gamma\pare{t}$, we prove in Lemma \ref{lemma002} that $y^\gamma\pare{t}$ is no more infinitesimal, when the time approaches $t_c$. Before stating the lemma we need to prove the following result which provides a control of the martingale terms appearing in \eqref{zaq7}.

For $\xi>0$, define
\begin{align}
A^\gamma_{\xi}:=\llav{\omega\in \Omega :\sup_{t\in[0,t_c]}\abs{\int_0^te^{-\pare{\beta-1}s}dM^{\gamma, y}(s)}\leq\gamma^{\frac{d}{2}-\xi}, \sup_{t\in[0,t_c]}\abs{M^{\gamma, x}(t)}\leq t_c^{\frac{1}{2}}\gamma^{\frac{d}{2}-\xi}}.
\end{align}
\begin{lemma}\label{poiu}
For every $\xi\in\pare{0, \frac{d}{2}}$,
\begin{align}\nonumber
\BB P\pare{A^\gamma_{ \xi}}\xrightarrow[\gamma\to 0]{}1.
\end{align}
\end{lemma}
\begin{proof}
Let $\pic{M^{\gamma, y}}(\cdot)$ be the compensator of $M^{\gamma, y}(\cdot)$ (we refer to \cite{HHK06} for the definition). By \eqref{4567}, through computations similar to \eqref{rana}, it is possible to show that 
\begin{align}\label{sads}
\pic{M^{\gamma, y}}(t)=\int_0^t\pare{L_\gamma\pare{y^\gamma\pare{s}}^2-2y^\gamma \pare{s}L_\gamma y^\gamma \pare{s}}ds\leq \pare{\beta+1}\gamma^d t,
\end{align}
for $t\geq 0$. Then by Doob's inequality and Ito's isometry we get
\begin{align}\label{m}
\begin{aligned}
\BB P\pare{\sup_{t\in[0,t_c]}\abs{\int_0^te^{-\pare{\beta-1}s}dM^{\gamma, y}(s)}>\gamma^{\frac{d}{2}-\xi}}&\leq\gamma^{-d+2\xi}\BB{E}_{\BB P}\pare{\int_0^{t_c}e^{-\pare{\beta-1}s}dM^{\gamma, y}(s)}^2\\
&= \gamma^{-d+2\xi}\BB{E}_{\BB P}\pare{\int_0^{t_c}e^{-2\pare{\beta-1}s}d\pic{M^{\gamma, y}}(s)}\\
&\leq C_1\gamma^{2\xi},
\end{aligned}
\end{align}
where $C_1$ is a constant depending on $\beta$. Let $C_2$ be the constant of Lemma \ref{ray} taking $G\equiv 1$, $J\equiv 1$, $T=t_c$ and $\zeta=t_c^{\frac{1}{2}}\gamma^{\frac{d}{2}-\xi}$, we obtain that
\begin{align}\label{pp}
\BB P\pare{\sup_{t\in[0,t_c]}\abs{M^{\gamma,x}(t)}>t_c^{\frac{1}{2}}\gamma^{\frac{d}{2}-\xi}}\leq C_2 \gamma^{2\xi}.
\end{align}
Inequalities \eqref{m} and \eqref{pp} complete the proof.
\end{proof}
For $T>0$, define
\begin{align}
\psi_1\pare{T}:=\frac{\beta}{\beta-1}e^{-\pare{\beta-1}T}-6\beta^3 e^{-2\pare{\beta-1}T},&\quad \psi_2\pare{T}:=\frac{\beta}{\beta-1}e^{-\pare{\beta-1}T}+6\beta^3 e^{-2\pare{\beta-1}T},\\
\psi_3\pare{T}:=e^{-\pare{\beta-1}T}-6\beta^3 e^{-2\pare{\beta-1}T},&\quad
\psi_4\pare{T}:=e^{-\pare{\beta-1}T}+6\beta^3 e^{-2\pare{\beta-1}T},
\end{align}
and call
\begin{align}\label{gammat}
\mathsf{\Gamma}_T:=\corch{1-\psi_2\pare{T}, 1-\psi_1\pare{T}}\times\corch{\psi_3\pare{T}, \psi_4\pare{T}}\cap \BB Q^2.
\end{align}
The following lemma holds.
\begin{lemma}\label{lemma002}
For every $T>0$, let $\Gamma^\gamma_{ T}:=\llav{\omega\in\Omega: \pare{x^\gamma\pare{t_c-T}, y^\gamma\pare{t_c-T}}\in \mathsf{\Gamma}_T}$, then
\begin{align}\label{moniasan}
\lim_{\gamma\to 0}\BB P\pare{\Gamma^\gamma_T}=1.
\end{align}
\end{lemma}
\begin{proof}
Fix $\tilde\xi\in\pare{0,\frac{d}{2}-\alpha}$ and consider the set $ A_{\tilde\xi}^\gamma$. Since Lemma \ref{poiu} holds, to conclude it is enough to show that 
\begin{align}\label{sad}
A^\gamma_{\tilde\xi}\subseteq \Gamma_T^\gamma,
\end{align}
for all $T>0$ and $\gamma$ small enough. Observe that $E^\gamma\pare{t}\geq 0$ then , fixing $\omega\in A^\gamma_{ \tilde\xi}$, we have that, for all $t\in [0,t_c]$,
\begin{align}\label{pwer}
y^\gamma(t)&\leq e^{\pare{\beta-1}t}\pare{\gamma^\alpha+\gamma^{\frac{d}{2}-\tilde\xi}}.
\end{align}
On the other hand
\begin{equation}\label{pwe}
\begin{split}
1-x^\gamma(t)&\leq \gamma^{\alpha}+\beta\int_0^ty^\gamma(s)ds+t_c^{\frac{1}{2}}\gamma^{\frac{d}{2}-\tilde\xi}\\
&\leq\gamma^{\alpha}+\beta\pare{\gamma^{\alpha}+\gamma^{\frac{d}{2}-\tilde\xi}}\int_0^te^{\pare{\beta-1}s}ds+t_c^{\frac{1}{2}}\gamma^{\frac{d}{2}-\tilde\xi}\\
&\leq \gamma^{\alpha}+\frac{\beta}{\beta-1} e^{\pare{\beta-1}t}\pare{\gamma^\alpha+\gamma^{\frac{d}{2}-\tilde\xi}}+t_c^{\frac{1}{2}}\gamma^{\frac{d}{2}-\tilde\xi}.
\end{split}
\end{equation}
For every $T>0$, by \eqref{pwer} and \eqref{pwe} we get the following bounds
\begin{align}\label{princ00}
y^\gamma(t_c-T)&\leq e^{-\pare{\beta-1}T}+err\pare{\gamma},
\end{align}
and
\begin{align}\label{princ0}
x^\gamma(t_c-T)&\geq 1-\frac{\beta}{\beta-1}e^{-\pare{\beta-1}T}+err\pare{\gamma}.
\end{align}
where $err\pare{\gamma}$ is a term vanishing as $\gamma$ goes to $0$. Using that $\gamma^{\frac{d}{2}-\tilde\xi}\leq \gamma^{\alpha}$ in \eqref{pwer} and \eqref{pwe}, we obtain that, for all $t\in[0, t_c]$, 
\begin{align}\label{kate}
\begin{aligned}
E^\gamma(t)&\leq \frac{6\beta^2}{\beta-1}e^{2\pare{\beta-1}t}\gamma^{2\alpha}+2\beta e^{\pare{\beta-1}t}t_c^{\frac{1}{2}}\gamma^{\alpha+\frac{d}{2}-\tilde\xi}.
\end{aligned}
\end{align}
Since in the right hand side of \eqref{zaq7} there is a negative sign before the first integral term, the upper bound \eqref{kate} can be used in \eqref{zaq7} to get the following lower bound
\begin{align}\label{princ2}
\begin{aligned}
y^\gamma(t)&\geq e^{\pare{\beta-1}t}\gamma^\alpha-\int_0^t e^{\pare{\beta-1}(t-s)}\pare{\frac{6\beta^2}{\beta-1}e^{2\pare{\beta-1}s}\gamma^{2\alpha}+ 2\beta e^{\pare{\beta-1}s}t_c^{\frac{1}{2}}\gamma^{\alpha+\frac{d}{2}-\tilde\xi}}ds-e^{\pare{\beta-1}t}\gamma^{\frac{d}{2}-\tilde\xi}\\
&\geq e^{\pare{\beta-1}t}\pare{\gamma^\alpha-\frac{6\beta^2}{\pare{\beta-1}^2}e^{\pare{\beta-1}t}\gamma^{2\alpha}- 2\beta t_c^{\frac{1}{2}}t\gamma^{\alpha+\frac{d}{2}-\tilde\xi}-\gamma^{\frac{d}{2}-\tilde\xi}}
\end{aligned}
\end{align}
and consequently
\begin{align}
y^\gamma\pare{t_c-T}\geq e^{-\pare{\beta-1}T}-\frac{6\beta^2}{\pare{\beta-1}^2}e^{-2\pare{\beta-1}T}+err\pare{\gamma}.
\end{align}
Since it appears a negative sign before the first integral in the right hand side of \eqref{zaq6}, by the lower bound \eqref{princ2} we get the following upper bound
\begin{align}\label{princ32}
\begin{aligned}
x^\gamma(t_c-T)&=1-\gamma^\alpha-\beta\int_0^{t_c-T} y^\gamma(s)ds+\int_0^{t_c-T}E^\gamma(s)ds+M^{\gamma, x}(t_c-T)\\
&\leq 1-\beta\int_0^{t_c-T}e^{\pare{\beta-1}s}\pare{\gamma^\alpha-\frac{6\beta^2}{\pare{\beta-1}^2}e^{\pare{\beta-1}s}\gamma^{2\alpha}- 2\beta t_c^{\frac{1}{2}}s\gamma^{\alpha+\frac{d}{2}-\tilde\xi}-\gamma^{\frac{d}{2}-\tilde\xi}}ds+\\
&+\int_0^{t_c-T}\pare{ \frac{6\beta^2}{\beta-1} e^{2\pare{\beta-1}s}\gamma^{2\alpha}+2\beta e^{\pare{\beta-1}s}t_c^{\frac{1}{2}}\gamma^{\alpha+\frac{d}{2}-\tilde\xi}}ds+err\pare{\gamma}\\
&=1-\frac{\beta}{\beta-1}e^{-\pare{\beta-1}T}+\frac{6\beta^3}{\pare{\beta-1}^3}e^{-2\pare{\beta-1}T}+err\pare{\gamma}.
\end{aligned}
\end{align}
Reminding that $x^\gamma\pare{\cdot}$ and $y^\gamma\pare{\cdot}$ take rational values, by \eqref{princ00}, \eqref{princ0},\eqref{princ2} and \eqref{princ32} we can conclude \eqref{sad}, for every $T>0$ and $\gamma$ small enough. 
\end{proof}
%
%
%
The latter result shows that, at the time $t_c-T$, the process is within the set $\mathsf{\Gamma}_T$, away from the unstable equilibrium point $\pare{1,0}$. This will allow to study what happens after $t_c-T$ using the results of section \eqref{rem1} once we prove that it is possible to approximate our Markov process with its deterministic version (given by the solution of system \eqref{222}).
\vspace{+10pt}

We can now proceed in the proof of \eqref{nahg}. Fix $x^*\in \pare{\hat x_\infty, 1}$, then by Markov property and Lemma \ref{lemma002}, for every $T>0$,
\begin{align}\label{mng}
\BB P\pare{S_{\gamma, x^*}<+\infty}&=\sum_{\pare{a,b}\in \mathsf{\Gamma}_T}\BB P\pare{S_{\gamma, x^*}<+\infty|x^\gamma\pare{t_c-T}=a, y^\gamma\pare{t_c-T}=b}\\
\vspace{-1700pt}
&\blanco{aaaaaaaaaaaaaaisddssssss}\times\BB P\pare{x^\gamma\pare{t_c-T}=a, y^\gamma\pare{t_c-T}=b}+err\pare{\gamma}\\
&=\sum_{\pare{a,b}\in \mathsf{\Gamma}_T}\BB P^\gamma_{a,b}\pare{\tilde S_{\gamma, T, x^*}<+\infty}\BB P\pare{x^\gamma\pare{t_c-T}=a, y^\gamma\pare{t_c-T}=b}+err\pare{\gamma},
\end{align}
where $\tilde S_{\gamma, T, x^*}:=S_{\gamma, x^*}-t_c+T$ and $\BB P^\gamma_{a,b}$ is the probability measure on $\pare{\Omega, \mathfrak{F}}$ defined as
\begin{align}\label{poca}
\BB P_{a,b}^\gamma\pare{A}=\BB P\pare{A|x^\gamma\pare{0}=a, y^\gamma\pare{0}=b},\quad\forall A\in \mathfrak{F}.
\end{align}

In order to conclude \eqref{nahg}, it is enough to show that there exists $T$  such that, for every $\gamma$ small enough,
\begin{align}\label{abb1}
\BB P^\gamma_{a,b}\pare{\tilde S_{\gamma, T, x^*}<+\infty}>1-\gamma^{\frac{d}{2}},\quad \forall \pare{a,b}\in \mathsf{\Gamma}_T.
\end{align}


For every $x_0, y_0, \bar x\in [0,1]$, we define the time
\begin{align}\label{poca2}
S_{x_0, y_0, \bar x}:=\inf\llav{t\geq 0:x_{x_0,y_0}\pare{t}=\bar x},
\end{align}
where $\pare{x_{x_0, y_0}\pare{t}, y_{x_0, y_0}\pare{t}}$ is the solution of system \eqref{222}, starting from $x(0)=x_0$ and $y(0)=y_0$.
$S_{x_0,  y_0, \bar x}$ is well defined if and only if $x_{x_0, y_0}\pare{\infty}<\bar x\leq x_0$, where $\dis{x_{x_0, y_0}\pare{\infty}:=\lim_{t\to +\infty}x_{x_0, y_0}\pare{t}}$.

Fix $T>0$ and $\pare{a,b}\in \mathsf{\Gamma}_T$, by \eqref{222} and \eqref{445} we get that
\begin{align}\label{kdfxx}
S_{a, b, x^*-e^{-T}}&=\int_{x^*-e^{-T}}^a\frac{1}{\beta x y\pare{x}}dx\\
&=\int_{x^*-e^{-T}}^a\frac{1}{\beta x \pare{-x+\frac{1}{\beta}\log x+a+b-\frac{1}{\beta}\log a}}dx\\
&\leq \int_{x^*-e^{-T}}^{1-\psi_1\pare{T}}\frac{1}{\beta x \pare{1-x+\frac{1}{\beta}\log x-\psi_2\pare{T}+\psi_3\pare{T}-\frac{1}{\beta}\log \pare{1-\psi_1\pare{T}}}}dx.
\end{align}
By studying the denominator of the integrand appearing in the right hand side of \eqref{kdfxx}, it is possible to prove that there exists $T_1:=T_1\pare{x^*}>0$, depending on $x^*$, such that, for all $T>T_1$ fixed and for all $x\in \pare{x^*-e^{-T}, 1-\psi_1\pare{T}}$ the following holds
\begin{align}
1-x+\frac{1}{\beta}\log x-\psi_2\pare{T}+\psi_3\pare{T}-\frac{1}{\beta}\log \pare{1-\psi_1\pare{T}}>0
\end{align}
and consequently the right hand side of \eqref{kdfxx}, which we denote by $\mathcal{L}^{\text{max}}_{T, x^*}$, is finite. In other words, for $T>T_1$,
\begin{align}\label{mok2}
S_{a, b, x^*-e^{-T}}\leq \mathcal{L}^{\text{max}}_{T, x^*}<+\infty, \quad\forall \pare{a,b}\in \mathsf{\Gamma}_T.
\end{align}
Inequality \eqref{abb1} follows from \eqref{mok2} once we prove that, for $T>T_1$ fixed and for every $\gamma$ small enough,
\begin{align}\label{lkof}
\BB P^\gamma_{a,b}\pare{\tilde S_{\gamma, T, x^*}\leq S_{a,b,x^*-e^{-T}}}>1-\gamma^{\frac{d}{2}}, \quad\forall \pare{a,b}\in \mathsf\Gamma_T.
\end{align}
Since $x^\gamma\pare{t}$ decreases in $t$, in order to prove \eqref{lkof} it is sufficient to prove that,
\begin{align}\label{ldd5}
\BB P_{a,b}^\gamma\pare{x^\gamma\pare{S_{a,b,x^*-e^{-T}}}\leq [x^*]_\gamma}>1-{\gamma^{\frac{d}{2}}}, \quad\forall \pare{a,b}\in \mathsf\Gamma_T.
\end{align}
For $T>0$ and $\pare{a,b}\in \mathsf{\Gamma}_T$, define
\begin{align}\label{kkz2}
\mathscr A_{a,b,T, x^*}^\gamma=\llav{\sup_{t\in [0, \CAL L^{\text{max}}_{T, x^*}]}\abs{x^\gamma\pare{t}-x_{a,b}\pare{t}}+\abs{y^\gamma\pare{t}-y_{a,b}\pare{t}}\leq C_{T, x^*}\gamma^{\frac{d}{4}}}
\end{align}
where $C_{T, x^*}=e^{\pare{2\beta+1}\CAL L^{\text{max}}_{T, x^*}}\pare{\mathcal{L}^{\text{max}}_{T, x^*}}^{\frac{1}{2}}$.
The proof of \eqref{ldd5} follows by showing that, for every $T>T_1$ fixed,
\begin{itemize}
\item[(1)] $\dis{\BB P^\gamma_{a,b}\pare{\mathscr A_{a,b,T, x^*}^\gamma}>1-\gamma^{\frac{d}{2}},\quad \forall\pare{a,b}\in \mathsf{\Gamma}_T},$\label{bigb}
\item[(2)] for every $\gamma$ small enough,
\begin{align}
\mathscr A_{a,b,T, x^*}^\gamma\subseteq\llav{x^\gamma\pare{S_{a,b, x^*-e^{-T}}}\leq [x^*]_\gamma}, \quad\forall\pare{a,b}\in\mathsf{\Gamma}_T.
\end{align}
\end{itemize}

Following the same strategy used to prove Lemma \ref{ray} it is possible to show that, for $\#\in\llav{x,y}$, $T>0$,
\begin{align}\label{mok3}
\BB P^\gamma_{a,b}\pare{\sup_{t\in \corch{0, \mathcal{L}^{\text{max}}_{T, x^*}}}\abs{M^{\gamma, \#}\pare{t}}>\pare{\mathcal{L}^{\text{max}}_{T, x^*}}^{\frac{1}{2}}\gamma^{\frac{d}{4}}}\leq \gamma^{\frac{d}{2}},\quad \forall\pare{a,b}\in \mathsf\Gamma_T.
\end{align}
Since, for every $T>0$ and $\pare{a,b}\in \mathsf{\Gamma}_T$, we have that
\begin{align}
\abs{x^\gamma\pare{t}-x_{a,b}\pare{t}}&\leq \abs{x^\gamma\pare{0}-a}+\beta\int_0^t\Big({\abs{x^\gamma\pare{s}-x_{a,b}\pare{s}}+\abs{y^\gamma\pare{s}-y_{a,b}\pare{s}}\Big)}ds+\abs{M^{\gamma, x}\pare{t}},\\
\abs{y^\gamma\pare{t}-y_{a,b}\pare{t}}&\leq \abs{y^\gamma\pare{0}-b}+\int_0^t\Big({\beta\abs{x^\gamma\pare{s}-x_{a,b}\pare{s}}+\pare{\beta+1}\abs{y^\gamma\pare{s}-y_{a,b}\pare{s}}\Big)}ds+\abs{M^{\gamma, y}\pare{t}},
\end{align}
by Gr\"{o}nwall's inequality and \eqref{mok3}, it follows that,
\begin{align}\label{mok4}
\BB P^\gamma_{a,b}&\pare{\mathscr A_{a,b,T, x^*}^\gamma}> 1-\gamma^{\frac{d}{2}}, \quad\forall \pare{a,b}\in \mathsf{\Gamma}_T,
\end{align}
and consequently item (1) is proved. To prove item (2) fix $T>T_1$, $\pare{a,b}\in\mathsf{\Gamma}_T$ and $\omega\in \mathscr A_{a,b,T, x^*}^\gamma$. By \eqref{mok2} we have that
\begin{align}
\abs{x^\gamma\pare{S_{a,b, x^*-e^{-T}}}-\pare{x^*-e^{-T}}}\leq C_{T, x^*}\gamma^{\frac{d}{4}}
\end{align}
and consequently
\begin{align}
x^\gamma\pare{S_{a,b, x^*-e^{-T}}}\leq x^*-e^{-T}+C_{T, x^*}\gamma^{\frac{d}{4}},
\end{align}
which is less than $[x^*]_\gamma$ for $\gamma$ smaller than a value $\gamma_1\pare {x^*,T}$ depending on $x^*$ and $T$ and this concludes the proof of item (2).
Observe that, by \eqref{mok2} we also proved that, for $T>T_1$ and $\gamma<\gamma_1$,
\begin{align}\label{mokl}
\BB P^\gamma_{a,b}\pare{\tilde S_{\gamma, T, x^*}<\mathcal{L}^{\text{max}}_{T, x^*}}>1-\gamma^{\frac{d}{2}},\quad \forall \pare{a,b}\in \mathsf{\Gamma}_T.
\end{align}
\eqref{mokl} will be useful in the proof of Theorem \ref{004}.

We proceed now in the proof of \eqref{nahgg}.
Fix $\varepsilon >0$ and observe that, by Markov property and Lemma \ref{lemma002},
\begin{align}\label{luia}
\BB P&\pare{\abs{y^\gamma\pare{S_{\gamma, x^*}}-y^*}>\varepsilon}\\
&\leq \sum_{\pare{a,b}\in \mathsf\Gamma_T}\BB P^\gamma_{a,b}\pare{\abs{y^\gamma\pare{\tilde S_{\gamma, T, x^*}}-y^*}>\varepsilon}\BB P\pare{x^\gamma\pare{t_c-T}=a, y^\gamma\pare{t_c-T}=b}+err\pare{\gamma}
\end{align}
for all $T>0$.
As before, to conclude it is enough to prove that there exists $T>0$, such that, for every $\gamma$ small enough,
\begin{align}\label{luia1}
\BB P^\gamma_{a,b}\pare{\abs{y^\gamma\pare{\tilde S_{\gamma, T, x^*}}-y^*}>\varepsilon}\leq 2{\gamma^{\frac{d}{2}}}.
\end{align}
Fixing $T>T_1$ and $\pare{a,b}\in\mathsf{\Gamma}_T$, observe that 
\begin{align}\label{luia2}
\abs {y^\gamma\pare{\tilde S_{\gamma, T, x^*}}-y^*}\leq \abs {y^\gamma\pare{\tilde S_{\gamma, T, x^*}}-y_{a,b}\pare{\tilde S_{\gamma, T, x^*}}}+\abs {y_{a,b}\pare{\tilde S_{\gamma, T, x^*}}-y^*}.
\end{align}
Fix $\omega\in \mathscr A_{a,b,T, x^*}^\gamma\cap \llav{\tilde S_{\gamma, T, x^*}\leq\CAL L^{\text{max}}_{T, x^*}}$, it follows that
\begin{align}\label{luia3}
\abs{y^\gamma(\tilde S_{\gamma, T, x^*})-y_{a,b}(\tilde S_{\gamma, T, x^*})}\leq C_{T, x^*}\gamma^{\frac{d}{4}}.
\end{align}
On the other hand, by \eqref{445},
\begin{align}
y_{a,b}\pare{\tilde S_{\gamma, T, x^*}}&=-x_{a,b}\pare{\tilde S_{\gamma, T, x^*}}+\frac{1}{\beta}\log x_{a,b}\pare{\tilde S_{\gamma, T, x^*}}+a+b-\frac{1}{\beta}\log a\\
&=y^*+x^*-x_{a,b}\pare{\tilde S_{\gamma, T, x^*}}-\frac{1}{\beta} \log x^*+\frac{1}{\beta}\log x_{a,b}\pare{\tilde S_{\gamma, T, x^*}}+a+b-1-\frac{1}{\beta}\log a .
\end{align}
Since $x_{a,b}\pare{t}\geq \hat{x}_\infty$, for $t\geq 0$, and the logarithm is Lipschitz in the interval $[\hat{x}_\infty, +\infty)$, we have that
\begin{align}\label{luia4}
\abs{y_{a,b}(\tilde S_{\gamma, T, x^*})-y^*}&\leq C\abs{x^*-x_{a,b}\pare{\tilde S_{\gamma, T, x^*}}}+C\pare{T}\\
&\leq C\pare{\abs{x^*-[x^*]_\gamma}+\abs{[x^*]_\gamma-x_{a,b}(\tilde S_{\gamma, T, x^*})}}+C\pare{T}\\
&\leq C\pare{\gamma+C_{T, x^*}\gamma^{\frac{d}{4}}}+C\pare{T},
\end{align}
where $C=1+\frac{1}{\beta\hat{x}_\infty}$ and $C\pare{T}:=\psi_4\pare{T}-\frac{1}{\beta}\log\pare{1-\psi_2\pare{T}}$ is such that $\dis{\lim _{T\to +\infty}C\pare{T}=0}$. Therefore there exists $T_2:=T_2\pare{\varepsilon}$ depending on $\varepsilon$, such that, for every $T>\max\llav{T_1, T_2}$ fixed and for all $\gamma$ smaller than a value  $\gamma_2:=\gamma_2\pare{x^*, T, \varepsilon}$, the right hand sides of \eqref{luia3} and \eqref{luia4} are smaller than $\frac{\varepsilon}{2}$ and $\frac{2\varepsilon}{3}$ respectively. Then
\begin{align}
\mathscr A_{a,b,T, x^*}^\gamma\cap \llav{\tilde S_{\gamma, T, x^*}<\CAL L_{T, x^*}}\subseteq \llav{\abs{y^\gamma\pare{\tilde S_{\gamma, T, x^*}}-y^*}<\varepsilon},
\end{align}
and consequently, by \eqref{mokl} and \eqref{mok4} we get \eqref{luia1}.
\subsection{Proof of Theorem \ref{004}}
By Markov property and Lemma \ref{lemma002}, for every $\CAL L^-, \CAL L^+, T>0$,
\begin{align}\label{poha}
\BB P&\pare{S_{\gamma, x^*}\in [t_c-T+\CAL L^-, t_c-T+\CAL L^+]}\\
&=\sum_{\pare{a,b}\in \mathsf{\Gamma}_T}\BB P^\gamma_{a,b}\pare{\tilde S_{\gamma, T, x^*}\in \corch{\CAL L^-, \CAL L^+}}\BB P\pare{x^\gamma\pare{t_c-T}=a, y^\gamma\pare{t_c-T}=b}+err\pare{\gamma};
\end{align}
we remind that $\tilde S_{\gamma, T, x^*}:=S_{\gamma, x^*}-t_c+T$ and that the definition of $\BB P^\gamma_{a,b}$ has been given in \eqref{poca}.
To prove \eqref{jaja} we will use \eqref{poha} replacing $\CAL L^+$ by $\CAL L_{T, x^*}^{\text{max}}$ (defined by the rhs of \eqref{kdfxx}) and $\CAL L^-$ by $\CAL L_{T, x^*}^{\text{min}}$ which will be defined later. Therefore
it is enough to show that there exists $T_{x^*}>0$ such that, for $T>T_{x^*}$ and for every $\gamma$ small enough,
\begin{align}\label{cri1}
\BB P^\gamma_{a,b}\pare{\tilde S_{\gamma, T, x^*}< \CAL L_{T, x^*}^{\text{max}}}>1-\gamma^{\frac{d}{2}},
\end{align}
and
\begin{align}\label{cri2}
\BB P^\gamma_{a,b}\pare{\tilde S_{\gamma, T, x^*}>\CAL L_{T, x^*}^{\text{min}}}>1-\gamma^{\frac{d}{2}}, 
\end{align}
for all $\pare{a,b}\in \mathsf{\Gamma}_T$.
Inequality \eqref{cri1} has already been showed in the proof of Lemma \ref{balu} for $T>T_1=T_1\pare{x^*}$ and $\gamma$ small enough (see \eqref{mokl}). To prove \eqref{cri2} we proceed in a  similar way. Fix $T>T_1$, $\pare{a,b}\in \mathsf \Gamma_{ T}$ and observe that,
\begin{align}\label{cri3}
S_{a,b, x^*+e^{- T}}\geq \int_{x^*+e^{- T}}^{1-\psi_2\pare{ T}}\frac{dx}{\beta x\pare{1-x+\frac{1}{\beta}\log x-\psi_1\pare{ T}+\psi_4\pare{ T}-\frac{1}{\beta}\log \pare{1-\psi_2\pare{T}}}}
\end{align}
We denote the right hand side of \eqref{cri3} by $\CAL L^{\text{min}}_{T, x^*}$. Inequality \eqref{cri2} follows once we show that, for every $\gamma$ small enough,
\begin{align}\label{cri4}
\BB P^\gamma_{a,b}\pare{\tilde S_{\gamma, T, x^*}\geq S_{a,b, x^*+e^{- T}}}>1-\gamma^{\frac{d}{2}}, \quad \forall \pare{a,b}\in \mathsf{\Gamma}_{ T}.
\end{align}
Since $x^\gamma\pare{t}$ decreases in $t$, \eqref{cri4} follows if we prove that
\begin{align}\label{ldd22}
\BB P_{a,b}^\gamma\pare{[x^*]_\gamma\leq x^\gamma\pare{S_{a,b, x^*+e^{- T}}}}>1-\gamma^{\frac{d}{2}}, \quad\forall\pare{a,b}\in\mathsf{\Gamma}_{ T}.
\end{align}
Observe that for the monotonicity of $x^\gamma\pare{t}$ it follows that $S_{a,b, x^*+e^{- T}}<S_{a,b, x^*-e^{- T}}\leq \CAL L^{\text{max}}_{T, x^*}$. Consequently by fixing $\omega\in \CAL A_{a,b, T, x^*}^\gamma$, whose definition is given in \eqref{kkz2}, we have that
\begin{align}
x^\gamma\pare{S_{a,b, x^*+e^{- T}}}\geq x^*+e^{-T}+C_{ T, x^*}\gamma^{\frac{d}{4}},
\end{align}
which is greater than $[x^*]_\gamma$ for $\gamma$ smaller than a value $\gamma_1:=\gamma_1\pare{x^*, T}$. More precisely, for every $T>T_1$ fixed and for all $\gamma<\gamma_1$, 
\begin{align}
\mathscr A_{a,b, T, x^*}^\gamma\subseteq\llav{[x^*]_\gamma\leq x^\gamma\pare{S_{a,b, x^*+e^{-T}}}}, \quad\forall\pare{a,b}\in\mathsf{\Gamma}_{ T}.
\end{align}
By \eqref{mok4} we conclude \eqref{ldd22}.
To prove \eqref{polka} observe that
\begin{align}\label{nnnj}
\CAL L_{T, x^*}^\text{max}-\CAL L_{T, x^*}^\text{min}=&\int_{x^*-e^{-T}}^{x^*+e^{-T}}\frac{1}{\beta x\pare{1-x+\frac{1}{\beta}\log x+G_1\pare{T}}}dx\\
&+\int_{x^*+e^{-T}}^{1-\psi_2\pare{T}}\corch{\frac{1}{\beta x\pare{1-x+\frac{1}{\beta}\log x+G_1\pare{T}}}-\frac{1}{\beta x\pare{1-x+\frac{1}{\beta}\log x+G_2\pare{T}}}}dx\\
&+\int_{1-\psi_2\pare{T}}^{1-\psi_1\pare{T}}\frac{1}{\beta x\pare{1-x+\frac{1}{\beta}\log x+G_1\pare{T}}}dx,
\end{align}
where 
\begin{align}
G_1\pare{T}=-\psi_2\pare{T}+\psi_3\pare{T}-\frac{1}{\beta}\log\pare{1-\psi_1\pare{T}}, \quad G_2\pare{T}=-\psi_1\pare{T}+\psi_4\pare{T}-\frac{1}{\beta}\log\pare{1-\psi_2\pare{T}}.
\end{align}
The first integral in the rhs of \eqref{nnnj} clearly vanishes as $T\to +\infty$. The second integral is bounded above, definitely in $T$, by
\begin{align}
\pare{1-x^*}\frac{G_2\pare{T}-G_1\pare{T}}{x^*\pare{\frac{1}{\beta}\log x^*}^2+\pare{G_1\pare{T}+G_2\pare{T}}\pare{1-x^*+\frac{1}{\beta}\log\pare{1-\psi_2\pare{T}}}+G_1\pare{T}G_2\pare{T}},
\end{align}
which vanishes as $T\to +\infty$. The third integral appearing in \eqref{nnnj} is bounded above by
\begin{align}
\pare{\psi_2\pare{T}-\psi_1\pare{T}}\frac{1}{\beta\pare{\psi_1\pare{T}+\frac{1}{\beta}\log\pare{1-\psi_2\pare{T}}+\psi_3\pare{T}-\psi_2\pare{T}-\frac{1}{\beta}\log\pare{1-\psi_2\pare{T}}}}
\end{align}
which also vanishes as $T\to +\infty$ and this concludes the proof.

\subsection{Proof of Theorem \ref{051}}	\label{sez4}
\paragraph*{Proof of item (a).}
Let $\pare {x(t), y(t)}$ be the solution of the macroscopic system \eqref{222} starting from $\pare{\rho_0, \rho_1}$, with $\rho_1$ satisfying the hypothesis of item (a). By Corollary \ref{teoo} such solution converges, as $t\to+\infty$, to an equilibrium point $\pare{x_\infty, 0}$,  where $x_\infty<\frac{1}{\beta}$, as  we proved in section \ref{rem1}.  

The proof of Theorem \ref{051} follows if we show it for $\varepsilon$ in the interval $\pare{0,1}$. Fix $\varepsilon\in(0,1)$ and $\delta$ such that $0<\delta<\min\llav{1, 4\pare{\frac{1}{\beta}-x_{\infty}}}$, then there exists a time $T:=T\pare{\varepsilon, \delta}\geq 0$, depending on $\varepsilon$ and $\delta$, at which
\begin{align}\nonumber
y(T)\leq\varepsilon\delta,\qquad x(T)-x_\infty\leq \frac{\varepsilon}{2},\qquad x(T)<\frac{1}{\beta}-\frac{\varepsilon\delta}{4}.
\end{align}
Define
\begin{align}\nonumber
\begin{aligned}
\mathscr{G}_{T, \varepsilon, \delta}:=\corch{x\pare{T}-\frac{\varepsilon\delta}{4},\; x\pare{T}+\frac{\varepsilon\delta}{4}}\times \corch{y\pare{T}-\frac{\varepsilon\delta}{4},\; y\pare{T}+\frac{\varepsilon\delta}{4}}\cap\BB Q^2,
\end{aligned}
\end{align}
by Corollary \ref{teoo},
\begin{align}\label{k1}
\BB P\pare{\pare{x^\gamma\pare{T}, y^\gamma\pare{T}}\in \mathscr{G}_{T, \varepsilon, \delta}}\geq 1-err\pare{\gamma},
\end{align}
where $err\pare{\gamma}$ is a term vanishing as $\gamma$ converges to $0$. 
Observe that
\begin{align}
\abs{x^\gamma\pare{\infty}-x_\infty}&\leq x^\gamma\pare{T}-x^\gamma\pare{\infty}+\abs{x^\gamma\pare{T}-x\pare{T}}+x\pare{T}-x_\infty\\
&\leq x^\gamma\pare{T}-x^\gamma\pare{\infty}+\abs{x^\gamma\pare{T}-x\pare{T}}+\frac{\varepsilon}{2},
\end{align}
and consequently,
\begin{align}\label{asdq}
\BB P&\pare{\abs{x^\gamma\pare{\infty}-x_\infty}>\varepsilon}\leq \BB P\pare{x^\gamma\pare{T}-x^\gamma\pare{\infty}>\frac{\varepsilon}{4}}+\BB P\pare{\abs{x^\gamma\pare{T}-x\pare{T}}>\frac{\varepsilon}{4}}\\
&\leq \sum_{\pare{x,y}\in\mathscr G_{T, \varepsilon, \delta}}\BB P\pare{x^\gamma\pare{T}-x^\gamma\pare{\infty}>\frac{\varepsilon}{4}\Big| x^\gamma\pare{T}=x, y^\gamma\pare{T}=y}\BB P\pare{x^\gamma\pare{T}=x, y^\gamma\pare{T}=y}\\
&\blanco{aaaaaaaaaaaaaa}+\BB P\pare{\abs{x^\gamma\pare{T}-x\pare{T}}>\frac{\varepsilon}{4}\Big|\mathscr{G}_{T, \varepsilon, \delta}}\BB P\pare{\mathscr{G}_{T, \varepsilon, \delta}}+err\pare{\gamma}\\
&\leq \sum_{\pare{x,y}\in\mathscr G_{T, \varepsilon, \delta}}\BB P\pare{x^\gamma\pare{T}-x^\gamma\pare{\infty}>\frac{\varepsilon}{4}\Big| x^\gamma\pare{T}=x, y^\gamma\pare{T}=y}\BB P\pare{x^\gamma\pare{T}=x, y^\gamma\pare{T}=y}+err\pare{\gamma}\\
&=\sum_{\pare{x,y}\in\mathscr G_{T, \varepsilon, \delta}}\BB P^\gamma_{x,y}\pare{x^\gamma\pare{0}-x^\gamma\pare{\infty}>\frac{\varepsilon}{4}}\BB P\pare{x^\gamma\pare{T}=x, y^\gamma\pare{T}=y}+err\pare{\gamma}.
\end{align}
In the last step of \eqref{asdq} we used the Markov property. We also recall that $\BB P_{x, y}^\gamma$ is a probability measure defined on $\pare{\Omega, \mathfrak{F}}$ as
\begin{align}
\BB P_{x,y}^\gamma\pare{A}=\BB P\pare{A|x^\gamma\pare{0}=x, y^\gamma\pare{0}=y},\quad\forall A\in \mathfrak{F}.
\end{align}

%

Consider now the process $\pare{x^\gamma\pare{t}, y^\gamma\pare{t}}_{t\geq 0}$,  starting from $x^\gamma\pare{0}=x$, $y^\gamma\pare{0}=y$ and call $Y_1^\gamma\pare{t}:=\gamma^{-d}y^\gamma\pare{t}$. 

Observe that $\pare{Y_1^\gamma\pare{t}}_{t\geq 0}$ coincides with a Markov birth and death process, starting from $\gamma^{-d}y$ individuals, with random birth rate varying in time $\llav{\lambda_1\pare{t}}_{t\geq 0}:=\llav{\beta x^{\gamma}\pare{t}}_{t\geq 0}$ and constant death rate $\mu_1:= 1$. Let $N_1^\gamma\pare{0, +\infty}$ be the number of individuals generated by such process in the interval $\pare{0, +\infty}$, it follows that $x^\gamma(0)-x^\gamma(\infty)= \gamma^d N_1^\gamma\pare{0, +\infty}$.

Call ${N_2^\gamma}(0, +\infty)$ the number of individuals,  generated in $(0, +\infty)$, in a classic birth and death process $\pare{Y_2^\gamma\pare{t}}_{t\geq 0}$, having $\corch{\gamma^{-d}\pare{y(T)+\frac{\varepsilon\delta}{4}}}+1$\footnote{$[a]$ denotes the integer part of $a$.} individuals at time $0$, with birth rate $\lambda_2:=\beta\pare{x(T)+\frac{\varepsilon\delta}{4}}$ and death rate $\mu_2:= 1$.

If $y\leq y\pare{T}+\frac{\varepsilon\delta}{4}$, $\lambda_1\pare{t}\leq \lambda_2$ for all $t\geq 0$, and $\mu_1=\mu_2$, then  $N_1^\gamma\pare{0, +\infty}$ can be coupled with ${N_2^\gamma}(0, +\infty)$ in a common probability space $\pare{\tilde\Omega, \tilde{\mathscr{F}}, \tilde{\BB P}}$ in order to prove that
\begin{align}\label{sure}
N_1^\gamma\pare{0, +\infty}< N_2^\gamma(0,+\infty)
\end{align}
$\tilde {\BB P}$-almost surely.
Indeed, for $i\in\llav{1, \ldots, \gamma^{-d}y}$, call $s_i$ the $i$-th element present at time $0$ in the process $Y_1^\gamma\pare{\cdot}$, while for  $j\in\llav{1, \ldots, Y_2^\gamma\pare{0}}$, call $r_j$, the $j$-th element present at time $0$ in the process  $Y_2^\gamma\pare{\cdot}$. 
Observe that $Y_1^\gamma\pare{0}\leq Y_2^\gamma\pare{0}$ and call $\tau_i$ the stopping time at which $r_i$ dies; $\tau_i$ has an exponential distribution of parameter $1$. It is possible to construct the process $\llav{A_{1,i}^\gamma\pare{t}}_t$  counting the number of individuals generated by $s_i$, from $0$ up to the time at which $s_i$ dies, by thinning the process $\llav{A_{2, i}^\gamma\pare{t}}_{t\in [0, \tau_i)}$, which counts the number of individual generated by $r_i$ in the interval $(0, \tau_i)$ and with Poisson distribution of parameter $\lambda_2$. 
The number of points of $A_{2,i}^\gamma$ always dominate the number of points of the process $A_{1, i}^\gamma$ with smaller rate, then, \eqref{sure} follows.

Thus, denoting by $\BB E_{\tilde{\BB P}}$ the expectation respect to the probability measure $\tilde{\BB P}$, by \eqref{sure} we can conclude that, for every $\pare{x,y}\in \mathscr {G}_{T, \varepsilon, \delta}$ such that $\BB P\pare{x^\gamma(T)=x, y^\gamma(T)=y}\neq 0$,\footnote{We recall that $x^\gamma\pare{T}$ and $y^\gamma\pare{T}$ can only assume values which are multiples of $\gamma^d$ and consequently the process $Y_1^\gamma\pare{\cdot}$ always starts at time $0$ with an integer number of individuals.}
\begin{align}
\BB P^\gamma_{x,y}\pare{x^\gamma(0)-x^\gamma(\infty)>\frac{\varepsilon}{4}}&\leq\tilde{\BB P}\pare{\gamma^d N_2^\gamma(0,+\infty)>\frac{\varepsilon}{4}}\\
&\leq \frac{4\gamma^d}{\varepsilon}\mathbb{E}_{\tilde{\BB P}}\pare{N_2^\gamma(0,+\infty)}\\
&=\frac{4\gamma^d}{\varepsilon}\frac{[\gamma^{-d}\pare{y(T)+\frac{\varepsilon\delta}{4}}]+1}{1-\beta\pare{x(T)+\frac{\varepsilon\delta}{4}}}\label{ps1}\\
&\leq\frac{5\delta}{1-\beta\pare{1+\frac{\varepsilon\delta}{4}}}+err\pare{\gamma}.\label{lkh}
\end{align}
In \eqref{ps1} we used that $\BB E_{\tilde{\BB P}}\pare{N_2^\gamma(0,+\infty)}={Y}_2^\gamma(0)\frac{\mu_2}{\mu_2-\lambda_2}$ as $\lambda_2<\mu_2$ under our choise of $T$, $\delta$ and $\varepsilon$ (see \cite{K48} for more details). Inequality \eqref{lkh} follows from the fact that $x\pare{T}\leq 1$ and the assumption $y\pare{T}\leq \varepsilon\delta$. 
By \eqref{asdq} and \eqref{lkh}, we get that
\begin{align}
\BB P\pare{\abs{x^\gamma\pare{\infty}-x_\infty}>\varepsilon}\leq \frac{5\delta}{1-\beta\pare{1+\frac{\varepsilon\delta}{4}}}+err\pare{\gamma},
\end{align}
which vanishes after taking $\gamma$ and $\delta$ to $0$. This concludes the proof of (a).

\paragraph*{Proof of item (b).}
Also in this case it will be enough to make the proof for $\varepsilon$ in the interval $\pare{0,1}$.
Fix $\varepsilon\in\pare{0,1}$, and $\delta$ such that $0< \delta<\min\llav{2\pare{\frac{1}{\beta}-\hat{x}_\infty}, 1}$, and take $x^*:=x^*\pare{\varepsilon, \delta}\in \pare{\hat x_{\beta},1}$ such that $x^*-\hat{x}_\infty<\frac{\varepsilon\delta}{3}$. Since $1-\hat x_\infty+\frac{1}{\beta}\log \hat x_\infty=0$ and the logarithm is Lipschitz in the interval $[\hat x_\infty, +\infty)$, it follows that
\begin{align}\label{ventisette}
y^*:=y^*\pare{x^*}&=1-x^*+\frac{1}{\beta}\log x^*\\
&\leq x^*-\hat x_\infty+\frac{1}{\beta}\pare{\log x^*-\log \hat x_\infty}\\
&\leq\pare{1+\frac{1}{\beta \hat{x}_\infty}}\varepsilon \delta.
\end{align}.

Despite $x^*$ and $y^*$ depend on $\varepsilon$ and $\delta$, from now on we will omit writing such dependence just to make the notation easier. We have that
\begin{align}\label{marc3}
\abs{x^\gamma\pare{\infty}-\hat{x}_\infty}&\leq \abs{x^\gamma\pare{\infty}-[x^*]_\gamma}+\abs{[x^*]_\gamma-\hat{x}_\infty}\\
&\leq \abs{x^\gamma\pare{\infty}-[x^*]_\gamma}+\gamma+\frac{\varepsilon}{3}\\
&\leq \abs{x^\gamma\pare{\infty}-[x^*]_\gamma}+\frac{\varepsilon}{2},
\end{align}
for $\gamma$ small enough. Then to conclude it is enough to prove that
\begin{align}\label{ds5}
\lim_{\gamma\to 0}\BB P\pare{\abs{x^\gamma\pare{\infty}-[x^*]_\gamma}>\frac{\varepsilon}{2}}=0.
\end{align} 
For every $y\in \corch{y^*-\varepsilon\delta, y^*+\varepsilon \delta}$, define
\begin{align}
\mathscr{S}^\gamma_{x^*,y}=\llav{S_{\gamma, x^*}<+\infty, y^\gamma\pare{S_{\gamma, x^*}}=y},
\end{align}
and observe that, by the Markov property and Theorem \ref{balu},
\begin{align}\label{pokl}
\BB P\pare{\abs{x^\gamma\pare{\infty}-[x^*]_\gamma}>\frac{\varepsilon}{2}}&=\sum_{y\in [y^*-\varepsilon\delta, y^*+\varepsilon\delta]\cap \BB Q}\BB P\pare{\abs{x^\gamma\pare{\infty}-[x^*]_\gamma}>\frac{\varepsilon}{2}\Big|\mathscr{S}^\gamma_{x^*,y}}\BB P\pare{\mathscr{S}^\gamma_{x^*,y}}\\
&\blanco{aaaaaaaaaaaaaaaaaaaaaaaaaaaaaaaaaaaaaaaaa}+err\pare{\gamma}\\
&=\sum_{y\in [y^*-\varepsilon\delta, y^*+\varepsilon\delta]\cap \BB Q}\BB P^\gamma_{[x^*]_\gamma, y}\pare{\abs{x^\gamma\pare{\infty}-[x^*]_\gamma}>\frac{\varepsilon}{2}}\BB P\pare{\mathscr{S}^\gamma_{x^*,y}}+err\pare{\gamma}.
\end{align}
Since \eqref{ventisette} holds, for every $y\in \corch{y^*-\varepsilon\delta, y^*+\varepsilon\delta}$ we have that $y\leq\pare{2+\frac{1}{\beta\hat x_\infty}}\varepsilon\delta$, while $\dis{[x^*]_\gamma<x^*+\gamma<\hat x_\infty+\frac{\varepsilon\delta}{2}}$ for every $\gamma<\gamma_1\pare{\varepsilon,\delta}:=\frac{\varepsilon\delta}{6}$.

Consequently, following the same strategy just used to prove the previous item, i.e. by coupling the system with a birth and death process with birth rate $\beta\pare{\hat{x}_\infty+\frac{\varepsilon\delta}{2}}$ and death rate $1$ and starting from $\gamma^{-1}\pare{2+\frac{1}{\beta\hat{x}_\infty}}\varepsilon\delta$ individuals, it is possible to show that, 
\begin{align}
\BB P^\gamma_{[x^*]_\gamma, y}\pare{\abs{x^\gamma\pare{\infty}-[x^*]_\gamma}>\frac{\varepsilon}{2}}\leq \frac{\pare{4+\frac{2}{\beta\hat{x}_\infty}}}{1-\beta\hat{x}_\infty}\delta, 
\end{align}
for all $\gamma<\gamma_1$. By \eqref{pokl} and Theorem \ref{balu} we get that, for all $\gamma<\gamma_1$,
\begin{align}\label{bhy}
\BB P\pare{\abs{x^\gamma\pare{\infty}-[x^*]_\gamma}>\frac{\varepsilon}{2}}\leq C\delta+err\pare{\gamma},
\end{align}
where $C=\dis{\frac{\pare{4+\frac{2}{\beta\hat{x}_\infty}}}{1-\beta\hat{x}_\infty}}$.  The left hand side of \eqref{bhy} is independent of $\delta$, \eqref{ds5} follows by taking the limit of $\gamma$ and $\delta$ to $0$ in both sides of \eqref{bhy}.

\section{Appendix}
\paragraph*{Topology of the space}
Let $\CAL M_1^+:=\CAL M_1^+\pare{\BB T^d}$ be the space of all positive measures on $\BB T^d$ with mass bounded by $1$, endowed with the weak topology. According to Chapter 4 of \cite{KL} we can define a metric on $\mathcal{M}_1^+$ by introducing a dense countable family $\{f_k\}_{k\geq1}$ of continuous functions on $\mathbb{T}^d$ as follows
\begin{align}
d_{\CAL M_1^+}(\mu,\nu)=\sum_{k=1}^{\infty}\frac{1}{2^k}\frac{\big|\pic{\mu,f_k}-\pic{\nu,f_k}\big|}{1+\big|\pic{\mu,f_k}-\pic{\nu,f_k}\big|}\quad \forall \nu,\mu\in\mathcal{M}_1^+.
\end{align}
 The metric space $\pare{\CAL M_1^+, d_{\CAL{M}_1^+}}$ is Polish (see \cite{Bo} for istance). Fixing $T>0$  we denote by $\mathcal{D} := D\pare{[0,T], \mathcal{M}_1^+}$, the space of right continuous functions with left limits taking values in $\mathcal M_1^+$. We endow $\CAL D$ with the modified Skorohod metric $d_{\text{SK}}$, defined as
\begin{align}
d_{\text{SK}}\pare{\Pi,\tilde\Pi}=\inf_{\lambda\in\Lambda}\max\llav{\|\lambda\|,\sup_{t\in [0,T]}d_{\CAL M_1^+}\pare{\pi_{\lambda \pare{t}},\pi_t}},
\end{align}
for every $\Pi=\pare{\pi_t}_{t\in[0,T]}, \tilde\Pi=(\tilde \pi_t)_{t\in[0,T]}\in\CAL D$ , where $\Lambda$ is the set of strictly increasing continuous functions $\lambda:[0,T]\to [0,T]$ such that $\lambda\pare{0}=0$ and $\lambda\pare{T}=T$, while
\begin{align}\nonumber
\|\lambda\|:=\sup_{s\neq t}\abs{\log\frac{\lambda(t)-\lambda(s)}{t-s}}.
\end{align}
By Proposition 4.1.1 of \cite{KL} the metric space $\pare{\CAL D, d_{\text{SK}}}$  is Polish. $\CAL D^2$ is endowed with the product topology and the distance 
\begin{align}
d\pare{\underline\Pi, \underline{\tilde\Pi}}=d_{\text{SK}}\pare{\Pi^0, \tilde\Pi^0}+d_{\text{SK}}\pare{\Pi^1, \tilde\Pi^1},
\end{align}
for every $\underline\Pi=\pare{\Pi^0, \Pi^1}, \underline{\tilde\Pi}=\pare{\tilde\Pi^0, \tilde\Pi^1}\in\CAL D^2$.

\bigskip
The following theorem guarantees the tightness of the sequence $P^\gamma$ defined in Section \ref{sez2}.
\begin{theorem}\label{tightness}
The sequence $P^\gamma$ is tight.
\end{theorem}
\begin{proof}
The tightness of the sequence $P^\gamma$ follows once we prove the tightness of the sequences of the correspondent marginals $P^{\gamma,0}$ and $P^{\gamma,1}$. We only prove it for $P^{\gamma, 0}$ as the other case is analogous. According to chapter 4 of \cite{KL}, to conclude, it is sufficient to show that the following conditions hold
\begin{enumerate}
\item $\forall t\in [0,T]$ and $\forall \epsilon>0$ there exist a compact set $K(t,\epsilon )\subset \CAL M_1^+$ such that $$\sup_{\gamma} P^{\gamma, 0}\pare{\Pi:\pi_t\not\in K\pare{t,\epsilon}}\leq \epsilon,$$
\item $\forall \epsilon>0$ $$\lim_{\zeta\to 0}\limsup_{\gamma\to 0} \sup_{\tau\in I_T\atop\theta\leq\zeta}  P^{\gamma,0}\pare{\Pi: d_{\CAL M_1^+}\pare{\pi_{\tau+\theta}, \pi_\tau}>\epsilon}=0,$$
where $I_T$ is the family of all stopping times bounded by $T$.
\end{enumerate}
The first condition is trivially satisfied because $\CAL M_1^+$ is compact (see \cite{Bo}). To prove the second condition, fix $\tau\in I_T$, $\varepsilon>0$ and observe that
\begin{align}
 P^{\gamma,0}\pare{\Pi: d_{\CAL M_1^+}\pare{\pi_{\tau+\theta}, \pi_\tau}>\epsilon}=P\pare{\omega:d_{\CAL M_1^+}\pare{\pi_{\tau+\theta}^{\gamma, 0}, \pi_\tau^{\gamma, 0}}>\varepsilon}.
\end{align}
Choosing $N\in\BB N$ such that $\ds{\sum_{k=N+1}^{+\infty}\frac{1}{2^k}<\frac{\varepsilon}{2}}$ we have that
\begin{align}\label{cvb}
d_{\CAL M_1^+}\pare{\pi_{\tau+\theta}^{\gamma, 0}, \pi_\tau^{\gamma, 0}}\leq\sum_{k=1}^N\abs{\pic{\pi_{\tau+\theta}^{\gamma, 0}, f_k}-\pic{\pi_\tau^{\gamma, 0}, f_k}}+\frac{\varepsilon}{2}.
\end{align}
For every $k\in \BB N$,  by \eqref{5} we get
\begin{align}\label{buongiorno}
\begin{aligned}
\abs{\pic{\pi_{\tau+\theta}^{\gamma, 0}, f_k}-\pic{\pi_\tau^{\gamma, 0}, f_k}}&\leq\abs{\int_\tau^{\tau+\theta}\beta\gamma^d \sum_{x\in \BB T_\gamma^d}\mathbb{I}_{\{\eta^\gamma_s(x)=0\}}\gamma^d\sum_{y\in \BB T_\gamma^d}\mathbb{I}_{\{\eta^\gamma_s(y)=1\}}J\pare{\gamma x, \gamma y}f_k(\gamma x)ds}\\
&+\abs{M_{\tau+\theta}^{\gamma, 0,f_k}-M_{\tau}^{\gamma, 0,f_k}}\\
&\leq C_k\theta+\abs{M_{\tau+\theta}^{\gamma, 0,f_k}-M_{\tau}^{\gamma, 0,f_k}},
\end{aligned}
\end{align}
where $C_k=\beta\|J\|_\infty \|f_k\|_\infty$. By \eqref{4567} and \eqref{rana} we get that, for every $\tau\in I_T$,
\begin{align}\label{dart}
\begin{aligned}
\BB E_{ P}&\pare{M_{\tau+\theta}^{\gamma, 0,f_k}-M_\tau^{\gamma, 0,f_k}}^2\\
&=\mathbb{E}_{P}\pare{\int_\tau^{\tau+\theta}\beta\gamma^{3d}\sum_{x\in \BB T_\gamma^d}\sum_{y\in \BB T_\gamma^d}J\pare{\gamma x, \gamma y}\mathbb{I}_{\llav{\eta^\gamma_s(y)=1}}\mathbb{I}_{\llav{\eta^\gamma_s(x)=0}}f_k\pare{\gamma x}^2ds}\\
&\leq C_k\gamma^d\theta.
\end{aligned}
\end{align}
Let 
\begin{align}
\mathscr A_{\gamma, N}=\llav{\omega\in\Omega : \abs{M_{\tau+\theta}^{\gamma, 0,f_k}-M_\tau^{\gamma, 0,f_k}}\leq\gamma^\frac{d}{4}, \;\forall k\in\llav{1, \ldots, N}},
\end{align}
by $\eqref{dart}$, using Chebyshev inequality, we can prove that $P\pare{\mathscr A_{\gamma, N}}\to 1$ as $\gamma$ vanishes. To conclude condition $2$, it is enough to show that $\mathscr A_{\gamma, N}\subset \llav{d_{\CAL M_1^+}\pare{\pi_{\tau+\theta}^{\gamma, 0}, \pi_\tau^{\gamma, 0}}\leq \varepsilon}$ for $\gamma$ and $\theta$ sufficiently small. Set $\omega\in \mathscr A_{\gamma, N}$, by \eqref{cvb} and \eqref{buongiorno} we obtain that 
\begin{align}
d_{\CAL M_1^+}\pare{\pi_{\tau+\theta}^{\gamma, 0}, \pi_\tau^{\gamma, 0}}\leq N\max_{k\in\llav{1\ldots N}}\llav{C_k}\theta+N\gamma^\frac{d}{4}+\frac{\varepsilon}{2},
\end{align}
which is less than $\varepsilon$ for $\gamma$ and $\theta$ sufficiently small.
\end{proof}

\bigskip

{\bf Acknowledgments.}
I am deeply indebted to Inés Armend\'ariz, Anna De Masi, Pablo Ferrari, Ida Germana Minelli, Errico Presutti, and Maria Eul\'alia Vares for many helpful discussions.
\bibliographystyle{amsalpha}
\bibliography{biblio}

\end{document}